\documentclass[11pt, oneside]{amsart}   	

\usepackage{geometry}                		
\usepackage{graphicx}				
\usepackage{caption}					
\usepackage{amsmath,amssymb,amsthm}
\usepackage{color}
\usepackage{mathrsfs}
\usepackage{pgfplots}

\pgfplotsset{compat=1.13}
\usetikzlibrary{shapes,positioning,intersections,quotes}
\usepackage{hyperref}
\newtheorem{theorem}{Theorem}[section]

\newtheorem{proposition}{Proposition}[section]
\newtheorem{corollary}{Corollary}[theorem]
\newtheorem{lemma}[theorem]{Lemma}
\newtheorem{assumption}{Assumption}
\newtheorem{remark}{Remark}
\theoremstyle{definition}
\newtheorem{note}{Note}

\newcommand{\R}{\mathbb{R}}

\newcommand{\N}{\mathbb{N}}
\newcommand{\Z}{\mathbb{Z}}
\newcommand{\C}{\mathbb{C}}
\newcommand{\Sph}{\mathbb{S}}

\newcommand{\aaa}{\tilde{a}}

\newcommand{\CC}{\mathcal{C}}

\newcommand{\opL}{\mathcal{L}}

\newcommand{\Nn}{\mathcal{N}}

\newcommand{\doub}{\mathfrak{D}}
\newcommand{\markov}{\mathfrak{M}}

\newcommand{\euc}{\mathfrak{E}}

\newcommand{\diff}{\mathrm{d}}
\newcommand{\M}{\mathbb{M}}
\newcommand{\Exp}{\mathrm{Exp}}
\newcommand{\D}{N}
\renewcommand{\d}{d}

\newcommand{\W}{\mathbf{W}}
\newcommand{\PP}{\mathbf{P}}
\newcommand{\p}{\mathbf{p}}
\newcommand{\G}{\mathbf{G}}

\newcommand{\hi}{\hat{\imath}}
\newcommand{\hj}{\hat{\jmath}}

\DeclareMathOperator{\dist}{dist}
\renewcommand{\vec}[1]{\boldsymbol{\mathbf{#1}}}

\newcommand{\Manoa}{M\=anoa}
\newcommand{\Hawaii}{Hawai\kern.05em`\kern.05em\relax i }

\newcommand{\lp}{\left(}
\newcommand{\rp}{\right)}
\newcommand{\rad}{\delta}
\newcommand{\delr}{\dfrac{|\xi-\zeta|}{\rad}}
\newcommand{\mlswght}{w}
\newcommand{\subdom}{\widehat{\Omega}}
\newcommand{\subXi}{\widehat{\Xi}}
\newcommand{\subX}{\widehat{X}}

\newcommand{\weightkernel}{\mathcal{W}}

\newcommand{\centerCard}{J}
\newcommand{\polDim}{M}

\newcommand{\matern}{s}

\DeclareMathOperator{\MLS}{MLS}

\newcommand{\I}{\mathcal{I}}
\newcommand{\J}{\mathcal{J}}

\newcommand{\revision}[1]{#1}
\newcommand{\revisioncom}[1]{}

\newcommand{\fix}[1]{#1}
\newcommand{\fixcom}[1]{}

\title{Generalized local polynomial reproductions}
\author{Thomas Hangelbroek}
\address{Department of Mathematics, University of \Hawaii– \Manoa, 2565 McCarthy Mall,
Honolulu, HI 96822, USA}
\email{hangelbr@math.hawaii.edu}

\author{Christian Rieger}
\address{Philipps-Universit\"at Marburg, Department of Mathematics and Computer Science, \linebreak
Hans--Meerwein--Stra\ss{}e 6, 35032 Marburg, 
Germany}
\email{riegerc@mathematik.uni-marburg.de}

\author{Grady B. Wright}
\address{Boise State University, 1910 University Drive, 83725, Boise, Idaho, USA}
\email{gradywright@boisestate.edu}

\date{\today}

\begin{document}

\begin{abstract}
We present a general framework, treating Lipschitz domains in Riemannian manifolds, that provides conditions guaranteeing the existence of norming sets and generalized {\em local polynomial reproductions}---a powerful tool used in the analysis of various mesh-free methods and a mesh-free method in its own right. As a key application, we prove the existence of smooth local polynomial reproductions on compact subsets of algebraic manifolds in $\R^n$ with Lipschitz boundary. These results are then applied to derive new findings on the existence, stability, regularity, locality, and approximation properties of shape functions for a {\em coordinate-free} moving least squares approximation method on algebraic manifolds, which operates directly on point clouds without requiring tangent plane approximations.

\revision{There are two appendices: the first derives high order Markov inequalities for polynomials on algebraic manifolds and the second gives instructions for calculating the dimension of the space of degree $m$ polynomials restricted to a real algebraic variety.}
\end{abstract}
\maketitle
\section{Introduction}
\label{S_Introduction}

Let $\M $ be a smooth compact $d$-dimensional Riemannian manifold without boundary. 
In this  article, we 
 provide the existence of a smooth 
 generalization of the local polynomial reproduction on $\M$.
 Specifically, for certain
 bounded regions $\Omega\subset \M$,
and  function spaces $\varPi\subset C^{1}(\M)$
(which satisfy Assumptions \ref{A_function_space} and \ref{A_domain} described below)
and for sampled finite subsets $\Xi\subset \Omega$ that are sufficiently dense,
 we prove existence of a map $a:\Xi \times \Omega \to \R$  which
 satisfies the following properties
 \begin{itemize}
\item {\em $\varPi$-reproduction:} $(\forall p\in \varPi) (\forall z\in \Omega)$   $\sum_{\xi\in\Xi} a (\xi,z) p(\xi) = p(z)$ 
\item {\em stability:} $\sup_{z\in\Omega} \sum_{\xi\in \Xi} |a(\xi,z)|<2$
\item {\em locality:} $a(\xi,z)=0$ unless $\xi$ is near to $z$
\item {\em regularity:} if $\varPi\subset C^{\infty}(\M)$, then each $a(\xi,\cdot)\in C^{\infty}(\M)$.
\end{itemize}
As a consequence of this general result,
we demonstrate the existence of local polynomial reproductions
(i.e., $\varPi =\mathcal{P}_m(\R^N)$, 
polynomials of fixed degree at most $m$)
on Lipschitz domains 
in  algebraic manifolds $\M$ in $\R^{\D}$.

As in the archetypal Euclidean case, considered 
in \cite{JSW} and \cite{Wendland}, 
which forms a model for our results,
the key to the construction is to establish 
that $\Xi$ is a {\em norming set} (sometimes called an admissible mesh)
for the norm on $C(\Omega)$. 
In short, this means that the sampling operator $S:
C(\Omega)\to \ell_{\infty}(\Xi):
f\mapsto f|_{\Xi}$ is bounded below
by a fixed constant, 
independent of the cardinality
of $\Xi$. 

In the Euclidean setting, this relies heavily on Markov's inequality 
$$\|p'\|_{C([0,T])} \le \frac{2m^2}{T} \|p\|_{C([0,T])} $$ for algebraic polynomials
$p\in \mathcal{P}_m(\R)$.
For spherical harmonics on $\Sph^d$,  
Markov's inequality can be replaced by Videnskii's inequality \cite[Eqn.\ (2)]{V}.

Although  Markov inequalities exist in a variety of exotic contexts,
particularly for smooth  algebraic varieties 
in 
\cite{bos:etal:1998,bos:etal:2006}
and for  boundaries of convex sets in \cite{kroo:2002},
they are often global in nature, and unsuitable for adapting the Euclidean argument
developed in \cite{JSW} and \cite{Wendland}.
 We make a subtle but significant modification to
 the established Euclidean machinery
 to accommodate work on Riemannian manifolds.
 This involves Markov-like covariant derivative estimates
 in conjunction with a doubling estimate.
 These
 have been established for algebraic polynomials on algebraic manifolds  in \cite{bos:etal:1995} and \cite{fefferman1996local}, and
 likely exist for other systems of elementary functions on manifolds
 (see, for instance, Laplacian eigenfunctions in 
 \cite{Donnelly_Fefferman}).


 \subsection{Applications of stable polynomial reproductions}
 Local polynomial reproductions and related constructions on spheres provide a powerful tool for providing error estimates
 in scattered data approximation. 
 They have been used to obtain sampling inequalities,  first introduced by Madych and Potter \cite{MP},
 but more recently considered in  \cite{RZ,RW}.
 They are also crucial to a number of estimates in kernel and RBF approximation.
 Wu and Schaback, in \cite{WS}, use local polynomial reproductions  to estimate the {\em power function} $P_X$
 for RBF interpolation, which measures the norm of the interpolation error functional $f\mapsto f(x)-I_{X}f(x)$
 over a {reproducing} kernel Hilbert space associated with the RBF. 
DeVore and Ron, in  \cite{DevRon}, use local polynomial reproductions
 to get  kernel approximation results on $\R^d$. \revision{Videnskii's}\revisioncom{Referee 1: Minor correction 3} inequality has been
 used to provide a local spherical harmonic reproduction
 used in \cite{HesseL} and  on SO(3) in \cite{HS}.

\subsection{Moving Least Squares}
A specific motivating application for the theory presented in this article is moving least squares (MLS) approximation on algebraic manifolds.  
MLS has origins in work of Shepard
\cite{Shepard:1968} and was studied in the 1980s 
by Lancaster, Farwig, Salkauskas \cite{Lancaster:Salkauskas:1981, Farwig} and others.
For a set of points $\Xi$ in a domain $\Omega\subset \R^d$, an MLS approximant for the data  $\vec{y}=(y_{\xi})_{\xi \in \Xi}\subset \R$ takes the form   
\begin{align}\label{eq:mlsprimal}
    \MLS_{\Xi,\vec{y}}(\zeta):=p_{\zeta}^*(\zeta), \text{ where } p_{\zeta}^*:=\arg\min_{p \in \mathcal{P}_m(\R^d)} \sum_{\xi \in \Xi} \left(p(\xi) -y_{\xi}\right)^{2} \weightkernel(\xi,\zeta),
\end{align}
 $\zeta \in \Omega$ is a given  point and $\weightkernel:\Omega \times \Omega \to [0,\infty)$ is a given weight function.
Approximation theoretic results for $\Omega\subset \R^d$ were given, for example, in \cite{Levin} and \cite{Wendland}.

Our goal is to develop a MLS technique for algebraic manifolds $\Omega\subset \M$ that does not use intrinsic coordinates to $\Omega$ (i.e., coordinate-free) or tangent plane approximations. We then follow a similar approach of~\cite{Wendland}, which explicitly uses the kinds of local polynomial reproductions we develop here, to analyze the approximation properties of the method. 

MLS methods are used widely as mesh-free techniques for
solving partial differential equations (PDES)  (e.g.,\cite{Belytschko:etal:1994,Liu_MLS,Bayona,MSD,Trask}) and have been extended to spherical regions in \cite{W_sphere, HP}.
Recent works, such as \cite{Jones, SoberMLS}, have also extended MLS to problems on embedded manifolds. However, unlike
to our coordinate-free approach, these methods address the basic problem (\ref{eq:mlsprimal}) by projecting points onto tangent spaces.

\subsection{Outline}
In the next section, we 
introduce Assumptions \ref{A_function_space} and \ref{A_domain} 
which  guarantee a norming set property, among other useful properties; 
this is  a standard method for generating local polynomial reproductions,
which we consider in section \ref{S_Main} to generate a local $\varPi$-reproduction.
Unlike established constructions as in \cite{Wendland}, we 
investigate the smoothness of the basic functions, in addition to their stability and locality.
This follows the argument recently developed in the Euclidean setting in \cite{HR}.

In section \ref{S_MLS} we consider a constructive method for obtaining  local $\varPi$-reproductions based on MLS approximation. This has the advantage that derivatives of the local $\varPi$-reproduction
reproduce derivatives of functions in $\varPi$. We prove stability, smoothness and locality of this methodology.

Section \ref{S_Alg} treats the concrete application of local (restricted) polynomial reproductions on algebraic manifolds,
which employs analytic results from \cite{fefferman1996local}.
Numerical experiments for this setup are considered in section \ref{S_NE}.

\revision{Appendix \ref{S_Markov} adapts the method of Bos et al.
\cite{bos:etal:1995} to treat higher order Markov inequalities on algebraic manifolds. 
Appendix \ref{S_ideal} discusses the dimension
of $\mathcal{P}_m(\M) = \{p|_{\M} \mid p\in \mathcal{P}_m(\R^N)\}$ when $\M$ is an
algebraic variety.}\revisioncom{These sections have been added.}

\section{Geometric Background}
\subsection{Background}
\label{SS_Background} 
Throughout this paper, we assume $\M$ is a connected, $d$-dimensional Riemannian manifold.
We denote by 
{$T\M$}
the tangent bundle of $\M$, 
and by $T^{k_1}_{k_2}\M$ the vector
bundle of tensors with contravariant rank $k_1$ and covariant rank $k_2$ (in particular,
$T^{1}_{0}\M= T\M$ is the tangent bundle and $T^{0}_{1}\M= T^*\M$ is the cotangent bundle).
We will denote the fiber at $x\in \M$ by $T^{k_1}_{k_2}\M_x$.
In this article, we will be concerned primarily with covariant tensors.

The fact that $\M$ is a Riemannian manifold means that on each tangent space $T\M_x$ there
is an inner product. This extends by duality to each fiber $T^{k_1}_{k_2}\M_x$ of the tensor bundle: 
for instance,
for a covariant tensor 
$S\in T^{0}_{k}\M_x$, we have
\begin{equation}
\label{Tensor_norm}
\|S\|_{T^{0}_{k}\M_x} = \max_{(V_1,\dots,V_k)\in (T\M_x)^{k}\setminus\{0\}}
\frac{|S(V_1,\dots,V_k)|}{\|V_1\|_{T\M_x}\dots \|V_k\|_{T\M_x}}.
\end{equation}

\subsubsection{Tensor fields}
For a chart $(U,\phi)$ for $\M$ 
we get the usual vector fields $\frac{\partial}{\partial x^j}$
and forms $d x_j$ ($1\le j\le d$), which act as local bases for $T\M$ and $T^*\M$ over $U$.

These can be used to generate bases for tensor fields (i.e., sections of tensor bundles). In particular, for a 
covariant tensor field $T:\M \to T^{0}_{k}\M$, we have basis 
elements $dx_{\hi}:= dx_{i_1} \cdots  dx_{i_k}$ 
for $\hi=(i_1,\dots, i_k)\in \{1,\dots,d\}^k$, which
allow us to write $T$ in coordinates as
$T (x)= 
\sum_{\hi \in \{1,\dots,d\}^k} 
( T(x))_{\hi} dx_{\hi}$. 

Of particular interest is Riemannian metric tensor $g:\M\to T^{0}_{2}\M$, written in coordinates over $U$ 
as 
$g(x)= \sum_{i,j\le d } g_{ij} (x)dx_{i}dx_j$.
Similarly, the volume element on $\M$ is $\diff \mu= \sqrt{\det(g_{ij})}dx_1\dots dx_d$.

At a point $x\in U$,
the tangent space $T\M_x$
has  inner product $\langle T, S\rangle_{T\M_x} = \sum g_{ij} T_i S_j$,
where $T = \sum_{i=1}^d T_i \frac{\partial}{\partial x^i}$
and $S = \sum_{j=1}^d S_j \frac{\partial}{\partial x^j}$.
The   inner product 
on $T^{0}_{k}\M_x$  obtained from
(\ref{Tensor_norm}) can be expressed in coordinates as
 $$\langle  T,S\rangle_{T^{0}_{k}\M_x}
 =
 \sum_{\hi,\hj} g^{i_1j_1}(x) \cdots g^{i_kj_k}(x)  T_{\hi}S_{\hj},
$$ 
 where 
 $(g^{ij})$ is the matrix inverse of $(g_{ij})$,
 $T = \sum_{\hi} ( T)_{\hi} dx_{\hi}$. 
 and $S= \sum_{\hj} ( S)_{\hj} dx_{\hj}$
 (see  \cite[Eqn. (2.1)]{HNW}).

\subsubsection{Geodesics and exponential map}
We denote the Riemannian distance on $\M$ by 
$\dist:\M\times \M \to [0,\infty)$
or, to avoid confusion when multiple distances are in use, we use $\dist_{\M}$. 
Since $\M$ is connected, pairs of nearby points $x,y\in \M$
can be connected by a geodesic $\gamma:[a,b]\to \M$ with $\gamma(a) =x$, $\gamma(b)=y$
and $\dist(x,y) = \int_a^b \| \gamma'(t) \|_{T\M_{\gamma(t)}} \diff t$.

At every point $x\in\M$, the exponential map $\Exp_x: B(0,r_x)\to {U}\subset \M$ is a smooth diffeomorphism
defined on an open neighborhood 
of $0\in T_x\M$. It has the property that it preserves radial distances: for any $y= \Exp_x(v) \in {U}$,
$\dist(x,y)= |0-v|$.
In fact, the map $(x,v)\mapsto \Exp_x(v):\mathfrak{U}\to \M$ 
is defined and smooth on an open neighborhood  $\mathfrak{U}$ of the zero section in $T\M$.

Consequently, for any  compact set $\Omega\subset \M$, the quantity $r_{\Omega} = \min_{x\in\Omega} r_x$
is positive, and there exist constants $0<\Gamma_1\le \Gamma_2$ such that
for any $z\in \Omega$ and any $x,y\in B(0,\mathrm{r}_{\Omega})$, the metric equivalence
\begin{equation}
\label{m_e}
\Gamma_1 |x-y| \le \dist
(\Exp_{z}(x),\Exp_{z}(y)) \le \Gamma_2 |x-y|
\end{equation}
holds.

\subsubsection{Covariant differentiation}
The cotangent derivative $\nabla$ maps tensor fields of rank $(k_1,k_2)$ to tensor fields of rank $(k_1,k_2+1)$.
In particular, $\nabla^k$ maps  functions to  covariant  tensor fields of rank $(0,k)$.
We can use this to generate smoothness norms: for an open, bounded set  $U\subset \M$,
$$\|f\|_{C^k(\overline{U})}  := \max_{j\le k} \max_{x\in \overline{U}} \|\nabla^j f(x)\|_{T^{0}_{j}\M_x}.$$

If $\Omega\subset \M$  is compact, then  \cite[Lemma 3.2]{HNW} 
ensures that there are uniform constants $C_1,C_2
{>0}$  
so that the family of exponential maps
$\{\Exp_x: B_r\to \M\mid x\in\Omega\}$
provides local metric equivalences: 
for any open set $U$ with $\overline{U}\subset B(0,{r_{\Omega}})$, we have 
\begin{equation}\label{eq:expmap}
C_{1} \|u\circ\Exp_x\|_{C^k(\overline{U})}\le  
\|u\|_{C^k(\overline{\Exp_x(U)})}
\le C_{2} \|u\circ\Exp_x\|_{C^k(\overline{U})}.
\end{equation}
This is  an application of a more general result which treats metric equivalence of Sobolev norms. 
Although the full metric equivalence is not necessary for our purposes, another consequence is the following:
for every $x\in\Omega$ and any $U\subset B(x,r_{\Omega})$,
\begin{equation}\label{eq:Lp_equivalence}
C_{1} \|u\circ\Exp_x\|_{L_p(U)}\le    \| u\|_{L_p(\Exp_x(U))}\le C_{2} \|u\circ\Exp_x\|_{L_p(U)}.
\end{equation}
\label{S_Norming}
For a metric space $\Omega$
and a set $\Xi \subset \Omega$, 
define the sampling operator
$
	S_{\Xi}: 
 f \mapsto f|_{\Xi}
$
as a map from 
$C(\Omega)$ to $\ell_{\infty}(\Xi)$.
We note that
$
	\left\| S_{\Xi} \right\| 
	\le 1
$. 
A norming set 
for a subspace $\Pi\subset C(\Omega)$
is a subset $\Xi\subset \Omega$
so that $S_{\Xi}$ is bounded below.
Thus, finding a norming set $\Xi$ 
 is equivalent to developing a
 {\em Marcinkiewicz-type discretization} 
 for $p=\infty$, 
 as considered 
 in \cite[Eqn. 1.2]{dai2021entropy}.

Such norming sets have been developed for scattered approximation on spheres in \cite{JSW}. 
This has been expanded to treat subsets of $\Sph^{d}$ 
and Euclidean regions satisfying interior cone conditions
in \cite{Wendland}.
\revision{Our}\revisioncom{Referee 2: Minor Comment 1} strategy resembles the latter, although 
the use of the doubling property (\ref{F_Doubling}) simplifies the argument.
 
Our goal in the next section is to prove
a general norming set property:
for balls $B_1\subset B_2$ 
and sufficiently dense $\Xi\subset \Omega$,
the set $B_2 \cap \Xi$ is a norming set 
for $(\varPi,\|\cdot\|_{C(B_1\cap \Omega)})$.
In other words, 
when restricted to $\varPi$, the sampling operator 
$$ 
S_{\Xi}: 
(\varPi,\|\cdot\|_{C(B_1\cap \Omega)})
 \to 
\left( 
\R^{B_2\cap \Xi}, 
\| \cdot \|_{
    \ell_{\infty}( 
        B_2\cap \Xi)
    }
\right), 
\quad
p \mapsto p|_{B_2\cap\Xi}
 $$ 
 is bounded below.

\smallskip

\section{Local Markov property and preliminary results}
\label{S_LMP}
\subsection{Basic analytic and geometric assumptions}
To prove the general result, we assume the following about the function space $\varPi$.
\begin{assumption}

 \label{A_function_space}
We  assume that $\varPi\subset C^{1}(\M)$ is a finite dimensional space of functions 
which satisfy a doubling condition
and a Markov inequality. 
Namely, there exist constants 
$\doub$, $\markov$, 
and $r^{\sharp},>0$ so that for every $p\in  \varPi$, 
$0<r<r^{\sharp},$ and $x_0\in \M$,

\begin{align}
\label{F_Doubling}
\|  p\|_{C( B(x_0,2r))}
&\le
2^{\doub}
\|  p\|_{C(B(x_0,r))},
\\
\label{F_LM}
\|\nabla p\|_{C(B(x_0,r))}
&\le 
\frac{{\markov}}{r}
\|  p\|_{C( B(x_0,r))}.
\end{align}
\end{assumption}
We will discuss the applicability to restricted polynomials on algebraic manifolds later. 
For now, we point
out that such results have been shown in \cite{fefferman1996local}  with constants ${\doub}$ and ${\markov}$ which
depend on the polynomial degree.

In this section, we prove existence of a local $\varPi$-reproduction on compact subsets 
$\Omega\subset \M$ which satisfy an interior cone condition. 
 To set up the definition,
we define a basic Euclidean cone 
with parameters $r>0$, 
$0<\omega<\pi/2$ and $v\in \Sph^{\d-1}$, as
the set
$$C_{r,\omega,v} := 
\{\revision{z}\in \R^{d} \mid \|\revision{z}\|\le  r \text{ and } v^T \revision{z} \ge \cos(\omega)\}.
\revisioncom{Referee 2: Minor Comment 2}
$$
The interior cone condition for $\Omega\subset \M$ is similar to
a  Euclidean cone condition, but involves geodesic cones.

\begin{assumption} 
\label{A_domain}
We assume that $\Omega\subset \M$ is compact,
and that it satisfies an interior cone
condition of radius $r^{\flat}>0$ and aperture $\omega<\pi/2$:
for each point $z$ of $\Omega$, there is a direction
$v\in \Sph^{\d-1}$ so that 
$\Exp_z(C_{r^{\flat},\omega,v})\subset \Omega$.
\end{assumption}

We define the basic parameter, combining 
the injectivity radius $r_{\Omega}$,
threshold distances for the Markov and doubling inequalities
$r^{\sharp}$,
and the cone parameter $r^{\flat}$:
 $$r^{*}:= \min(r_{\Omega},r^{\sharp},r^{\flat}).$$
It follows
 that there exist constants $\alpha_{\Omega}\le \revision{\omega}_{\Omega}$\revisioncom{Referee 2: Minor Comment 3} 
 so that for any $x\in \Omega$ 
and $r\le r^*$,
\begin{equation}\label{ball_bound}
 \alpha_{\Omega}r^d \le \mu(B(x,r)\cap \Omega)\le \omega_{\Omega} r^d
 \end{equation}
 (this is a direct consequence of   (\ref{eq:Lp_equivalence}),with $p=1$ and $u= \chi_{C_{r,\omega,v}}$).

We make use of the following result, which follows easily 
from \cite[Lemma A.7]{HNW-poly} and
shows that cones contain interior balls.

\begin{lemma}
\label{ball_in_a_cone}
\revision{Suppose $\M$ is a Riemannian manifold,
  $\Omega\subset \M$  is compact
and $\CC=\Exp_z(C_{r^{\flat},\omega,v})\subset \Omega$.
Then for 
$0<\rho<r^{\flat}/(1+\sin \omega)$, 
and
 $\zeta=\Exp_z(\rho v)$, we have   
$\overline{B(\zeta,r)}\subset \CC$
where $r=\Gamma_1 \rho \sin(\omega)$ is determined by the cone aperture
and the constant of metric equivalence appearing in (\ref{m_e}).}\revisioncom{Referee 2: Minor Comment 4}
\end{lemma}
\begin{proof}
Note that 
$\rho\sin\omega\le r^{\flat}-\rho$ and
$r\le \Gamma_1(r^{\flat}-\rho)$. Thus, if $\xi\in B(\zeta,r)$, then 
$$\dist(z,\xi)< \rho+r \le
 \Gamma_1 r^{\flat} +(1-\Gamma_1)\rho< r^{\flat}.$$ Hence $B(\zeta,r)\subset B(z,r^{\flat})$.

For $\xi \in B (z,r^{\flat})\setminus \CC$, let $\xi = \Exp_{z}(t u)$ with $t<r^{\flat}$ and $u\in \Sph^{\d-1}$.
Then $tu\notin C_{r^{\flat},\omega,v}$, and hence 
$|tu -\rho v| > \rho \sin \omega$.
Thus $\dist(\xi,\zeta) \ge \Gamma_1 |tu-\rho v| > r$, so $\xi\notin \overline{B(\zeta,r)}$.
\end{proof}

For a point
$x_0\in \Omega$, we 
can use  Lemma \ref{ball_in_a_cone}
and the doubling property (\ref{F_Doubling})
 to control the size of
$\|p\|_{C(B(x_0,R))}$ 
by using 
$\|p\|_{C(B(x_0,R)\cap \Omega)}$.

\begin{lemma} 
\label{int_ext_sup}
Suppose $\varPi$ satisfies Assumption \ref{A_function_space}
and $\Omega\subset \M$ satisfies Assumption \ref{A_domain}.
Then for any $x_0\in \Omega$ and 
$R\le r^*$,
$$\|p\|_{C(\overline{B(x_0,R)})} \le
2^{\doub}  
\left( \frac{2+\sin \omega}{\Gamma_1\sin\omega}\right)^{\doub}
 \|p\|_{C(\overline{B(x_0,R)}\cap \Omega)}.$$
\end{lemma}
\begin{proof}
Fix $p\in \varPi$. 
Find $z_1\in\overline{B(x_0,R)}\cap \Omega$ 
and $z_0\in \overline{B(x_0,R)}$ (possibly different than $z_1$)
so that 
$|p(z_1)| = \|p\|_{C(\overline{B(x_0,R)}\cap \Omega)}$
and $|p(z_0)| = \|p\|_{C(\overline{B(x_0,R)})}$.

 By assumption, there is a cone 
 $\CC= \Exp_{x_0}(C_{R,\omega,v})$ 
 contained in 
 $\overline{B(x_0,R)}\cap \Omega$.
Let 
$\rho:=R/(1+ \sin(\omega)) $,
and  note that for $\zeta := \Exp_{x_0}(\rho v)$ and $r:= \Gamma_1 \rho \sin \omega $, we have
$$\|p\|_{C(\overline{B(\zeta,r)})} \le \|p\|_{C(\overline{B(x_0,R)}\cap \Omega)}\le |p(z_1)|$$ 
since
$\overline{B(\zeta,r)}$ 
is
contained in $\CC$ by Lemma \ref{ball_in_a_cone}.

At the same time, 
$\dist(z_0,\zeta) \le R+\rho$, 
so $z_0\in \overline{B(\zeta, R+\rho)}$. 
Find $j\in\N$ so that 
$
2^{j-1}\Gamma_1{ \sin \omega }
\le 
2+ \sin \omega
<
2^j
\Gamma_1
{ \sin \omega }$,
so $j$ depends only on $\Gamma_1$ and $\omega$.
By the definition of $\rho$, we have
$ R+\rho<2^j \Gamma_1 \rho\sin(\omega)   $,
so by applying the doubling inequality 
(\ref{F_Doubling}) $j$ times, 
we have that
%
$$
|p(z) |\le 
\|p\|_{C(\overline{B(\zeta, R+\rho)})} 
\le 
2^{j\doub}
\|p\|_{C(\overline{B(\zeta,r)})} 
\le 2^{\doub}  \left( \frac{2+\sin \omega}{\Gamma_1\sin\omega}\right)^{\doub}|p(z_1)|
$$
%
and the lemma follows.
\end{proof}

\subsection{ A lower bound}

Under assumptions \ref{A_function_space} and \ref{A_domain}, let us define
$$\kappa :=  {\markov} 2^{2\doub}
\left( \frac{2+\sin \omega}{\Gamma_1\sin\omega}\right)^{\doub} \frac{1+\Gamma_1 \sin\omega}{\Gamma_1 \sin \omega }$$
which will be used to describe the support radius of the norming
set.

\begin{lemma}
\label{L_main_ball_estimate}
 Suppose $\M$ is a Riemannian manifold, 
 $\varPi\subset C^1(\M)$ satisfies 
 Assumption \ref{A_function_space} 
 and 
 $\Omega\subset\M$ satisfies 
 Assumption \ref{A_domain}.
Then
for any 
 $x_0\in \Omega$, 
 $0<\epsilon<1$
 and
 $R<r^*/(1+\epsilon)$ 
  there is  a ball 
  $\tilde{B} := 
   \overline{ B(\zeta, \frac{\epsilon R}{\kappa })}
  \subset \Omega\cap \overline{ B(x_0,(1+\epsilon)R)}$
  so that for any  $p\in \varPi$,
\begin{align*}
 \min_{\xi\in\tilde{B}}|p(\xi)|
 \ge 
 (1-\epsilon)
 {\|p\|_{C(\overline{B(x_0,R)}\cap \Omega)}}
\end{align*}
holds.
\end{lemma}

\begin{proof}
Without loss of generality, let
$p \in\varPi$ satisfy $\|p\|_{C(\overline{B(x_0, R)}\cap\Omega)}=1$.
For any point $z\in \Omega\cap \overline{B(x_0,R)}$ where $|p(z)| =1$,
there is a cone $\CC = \Exp_{z}(C_{r^*,\omega,v})$ 
which is centered at $z$ and contained in $\Omega$.

By Lemma \ref{ball_in_a_cone},
there is a closed ball of radius 
$
\frac{\epsilon R}{\kappa }$
centered at 
$\zeta= 
\Exp_z(
\frac{ 
\epsilon R}{\kappa \Gamma_1 \sin\omega} v)$
contained in $\CC$. 
By the triangle inequality, 
if 
$\xi\in 
\overline{B(\zeta,
\frac{\epsilon R}{\kappa })}$, 
then 
$$ 
\dist(\xi,z)
\le
\frac{\epsilon}{\kappa}
\frac {1+\Gamma_1 \sin\omega}{ \Gamma_1 \sin\omega}\, R
\le \epsilon .
$$
By another application of the triangle inequality,
$\dist(\xi,x_0) \le
(1+{\epsilon})
R$ follows.

\begin{figure}[t]
\begin{tikzpicture}[thick,font=\sffamily\Large]
\scalebox{0.4}{
\rotatebox{70}{
\draw[very thick] (0,0) --  (5,4);
\draw[very thick] (0,0) --  (5,-4);
\draw[very thick] (0,0) --  (5,-4);
\draw[very thick] (5,-4) arc (-30:30:8) ;
\draw (3,0) circle (1.88cm) ;
\filldraw (3,0) circle[radius=1.5pt];
}
}
\node[above right=1pt of {(0.4,0.5)}] {$\zeta$};
\filldraw (0,0) circle[radius=1.5pt];
\node[below left=-2pt of {(0,0)}] {$z$};
\draw[color=black,domain=-2.5:3] 
  plot (\x,{(\x*\x*\x/7)-(3*\x)/8});
 \filldraw (-1.5,0.6) circle[radius=1.5pt];
\node[above left=-2pt of {(-1.5,0.6)}] {$x_0$}; 
\node[ left=-2pt of {(1.2,2)}] {$\mathcal{C}$};
\node[ left=-2pt of {(2,3)}] {$\Omega$};
\end{tikzpicture}
\vspace*{-15mm}
\caption{A diagram indicating the points $x_0$, $z$ and $\zeta$.\label{fig:CONE}}
\end{figure}
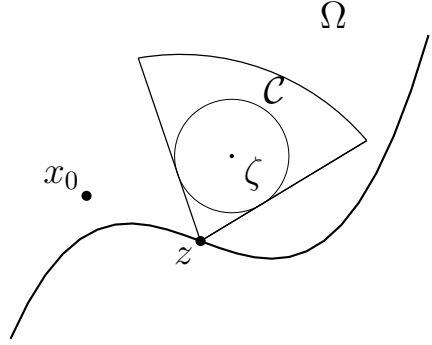

Thus  we may express 
$\xi \in 
\overline{B(\zeta,
\frac{ \epsilon R}{\kappa })}$ 
as
$\xi= \Exp_z( \dist(\xi,z) w)$ for some  vector 
$w\in \Sph^{\d-1}$.
Consider  the geodesic curve $\gamma:[0,\dist(\xi,z)]\to \M$ defined by
$\gamma(t) =\Exp_{z}
\Bigl( t w\Bigr)$.
Because the image of this curve lies is in 
$\overline{B(z, \dist(\xi,z))}$ 
which is contained in
$  \overline{B(x_0, (1+\epsilon)R)}$, 
and because $\|\gamma'(t)\|_{T_0^1\M_{\gamma(t)}}=1$, we have
\begin{align*}
	\left| p(\xi) - p(z) \right| 
& = 
   \left| \int_{0}^{\dist(\xi,z)} \nabla p \left(\gamma(t)\right) \gamma'(t) \diff t\right| \\
&\le 
  \dist(\xi,z)
\left\|\nabla p\right\|_{C(\overline{B(z,\dist(z,\xi))})}
 \\
&\le  
 \frac{\epsilon}{\kappa}
\frac {1+\Gamma_1 \sin\omega}{ \Gamma_1 \sin\omega}R
 \frac{\markov}{
(1+\epsilon)R }
 \|p\|_{C( \overline{B(x_0, 
(1+\epsilon) 
 R)})}  
\end{align*}
by (\ref{F_LM}). 
By the doubling inequality, we have
$  \|p \|_{C(\overline{B(x_0, 2R)})} \le  2^{\doub}  \|p\|_{C(\overline{B(x_0, R)})} $.
Applying Lemma \ref{int_ext_sup} we have
$\|p\|_{C(\overline{B(x_0,R)})}\le   
2^{\doub}\left( \frac{2+\sin \omega}{\Gamma_1\sin\omega}\right)^{\doub}
\|p\|_{C(\overline{B(x_0,R)}\cap \Omega)} $.
Recalling that  $z$ was chosen so that
$\|p\|_{C(\overline{B(x_0, R)}\cap \Omega)} = |p(z)|=1$, we have
$$
\bigl| |p(\xi) |- 1 \bigr| 
\le 
\frac{\epsilon}{1+\epsilon}
{\markov} 2^{2\doub}
\left( \frac{2+\sin \omega}{\Gamma_1\sin\omega}\right)^{\doub} 
\frac {1+\Gamma_1 \sin\omega}{\Gamma_1 \sin\omega}\frac{1}{\kappa}
 \le 
 \epsilon
.$$
Thus 
$
p(\xi)
\ge 
1-\epsilon
$
and the lemma follows.
\end{proof}



An immediate consequence of this lemma is the following norming set result, which
holds for subsets $\Xi\subset \M$ which are sufficiently well-distributed (or {\em sufficiently dense}) in $\Omega$, 
as measured
by the {\em fill-distance} $$h_{\Xi,\Omega} := \max_{x\in\Omega} \dist(x,\Xi),$$ which
we often abbreviate to $h=h_{\Xi,\Omega} $ when the context is clear. In this case,
being sufficiently well-distributed means that $h_{\Xi,\Omega}$ is less than some
given quantity which depends on Assumptions \ref{A_function_space} and \ref{A_domain}, as well as user-defined parameters. 
%
%
%
\begin{corollary}[Norming Set]
\label{C_norming}
 Suppose 
$\M$, $\varPi$ and $\Omega$ are
as in Lemma \ref{int_ext_sup}.
 If $\epsilon>0$, $R\le \frac{r^*}{1+\epsilon}$ and
 $\Xi \subset \Omega$ has fill distance 
 which  satisfies 
 $
 h:=h_{\Xi,\Omega}\le 
\epsilon  \frac{ R}{\kappa}
 $,
%
  then for any $x\in \Omega$ and $p\in \varPi$
  we have the inequality
\begin{align*}
  \max_{\xi\in \Xi\cap B\left(x,{(1+\epsilon)}R\right)}|p(\xi)|
 \ge 
 (1-\epsilon)
 {\|p\|_{C(\overline{B(x,R)}\cap \Omega)}}.
\end{align*}
\end{corollary}
A tiny but useful modification of this result employs $r:=\epsilon R$, in which case 
%
\begin{equation}
\label{norming_II}
  \max_{\xi\in \Xi\cap B\left(x,\frac{(1+\epsilon)}{\epsilon}r\right)}|p(\xi)|
 \ge 
 (1-\epsilon)
 {\|p\|_{C(\overline{B(x,\frac{r}{\epsilon})}\cap \Omega)}}
 \ge  
 (1-\epsilon)
 {\|p\|_{C(\overline{B(x,{r})}\cap \Omega)}}.
 \end{equation}
 %
 provided $r\le\frac{\epsilon}{1+\epsilon} r^* $ and $h\le \frac{r}{\kappa}$.
%
%
%
\begin{proof}
By  Lemma \ref{L_main_ball_estimate}, 
there is a ball  $\tilde{B}=\overline{B(\zeta,\frac{\epsilon R}{\kappa})}$ 
contained in  
$ {B(x,{(1+\epsilon)}R)}\cap \Omega$
where 
$|p(z)|\ge (1-\epsilon)
\|p\|_{C(B(x,{(1+\epsilon)}R)\cap \Omega)} $
for all $z\in \tilde{B}$.
By the assumption on $h$, 
 $\overline{B(\zeta,h)}\subset  \overline{B( \zeta, \frac{\epsilon R}{\kappa})} =\tilde{B}$ 
contains a point $\xi\in \Xi$.
\end{proof}

Another consequence of Lemma \ref{L_main_ball_estimate} is a local Nikolskii-type  inequality.
%
%
%
\begin{corollary}[Nikolskii]
\label{C_Nikolskii}
Suppose 
$\M$, $\varPi$ and $\Omega$ are
as in Lemma \ref{int_ext_sup}.
Then there is a constant $\mathfrak{N}$ so that  for any
$1\le q<\infty$,
 $x_0\in \Omega$ and 
$R\le r^*$,
$$\|p\|_{C(\overline{B(x_0,R)}\cap \Omega)} \le
\mathfrak{N} R^{-d/q}
 \|p\|_{L_q(B(x_0,R)\cap \Omega)}.$$
\end{corollary}
\begin{proof}
Let $R'=R/2$.
By applying Lemma \ref{L_main_ball_estimate} with $\epsilon = \frac12$, we
have that  
$$ 
\int_{\tilde{B}}|p(x)|^q\diff x
\ge 
\frac1{4}  \mathrm{vol}(\tilde{B})\|p\|_{C(B(x_0,R')\cap\Omega)}^q
$$
with $\tilde{B}$ a ball of radius $\frac{ R'}{\kappa } =\frac{ R}{2\kappa }$ contained in $B(x_0,2R')\cap \Omega$.
Thus, 
$$
\int_{B(x_0,R)\cap \Omega}|p(x)|^q\diff x
\ge 
\frac1{4}\alpha _{\M}
\Bigl( \frac{ R}{2\kappa} \Bigr)^d 
\|p\|_{C(B(x_0,R')\cap\Omega)}^q.
$$
By Lemma \ref{int_ext_sup} we have
\begin{eqnarray*}
 \|p\|_{C(B(x_0,2R')\cap \Omega)}
&\le& 2^{\doub}\|p\|_{C(B(x_0,R'))}\\
&\le&
 2^{2\doub}
\left( \frac{2+\sin \omega}{\Gamma_1\sin\omega}\right)^{\doub}
\|p\|_{C(B(x_0,R')\cap \Omega)}.
\end{eqnarray*}
and the result follows.
\end{proof}

\section{Main Result: local \texorpdfstring{$\varPi$}{varPi}-reproductions for \texorpdfstring{$\Omega$}{Omega}}
\label{S_Main}
We now prove existence 
of a local polynomial reproduction for sufficiently dense subsets $\Xi\subset \Omega$.  
This follows from a fairly standard argument
 using the  {\em norming set}
from Corollary \ref{C_norming}.
Note that throughout the remainder of the paper we define the threshold fill-distance
$$h_0:= \frac{\epsilon r^*}{(1+\epsilon)\kappa}$$
which depends on $0<\epsilon<1$ and  constants from Assumptions \ref{A_function_space} and \ref{A_domain}.

\subsection{Local Reproduction of functionals}
\revision{For $\alpha\in \Omega$, 
define $V_{\alpha}:=\varPi|_{(B(\alpha,\kappa h)\cap \Omega)}$.
Endow $V_{\alpha}$ with the norm
$$
\|p\|_{V_{\alpha}} 
:=
\sup_{z\in B(\alpha,\kappa h)\cap \Omega} |p(z)|
$$
and let $V_{\alpha}'$ denote the dual space
(i.e., the vector space of linear
functionals from $V_{\alpha}$ to $\R$)
equipped with norm 
$\lambda\mapsto \|\lambda\|_{V_{\alpha}'} 
= 
\sup_{p\in V_{\alpha}\setminus\{0\}} \frac{|\lambda p|}{\|p\|_{V_{\alpha}}}$.}\revisioncom{Referee 2: minor comment 5}

For $0<\epsilon<1$, the inequality (\ref{norming_II})  after
Corollary \ref{C_norming} (with $r=\kappa h$)
shows that  
the norm
$$\|p\|_{V_{\alpha}} :=\sup_{z\in B(\alpha,\kappa h)\cap \Omega} |p(z)|$$
is controlled above by $\|p|_{\Xi_{\alpha}}\|_{W_{\alpha,\epsilon} }$ where
$$
W_{\alpha,\epsilon} := \Bigl(\R^{\Xi_{\alpha,\epsilon}},\ell_{\infty}(\Xi_{\alpha,\epsilon})\Bigr)
\quad
\text{ and }
\quad 
\Xi_{\alpha,\epsilon}:=B\Bigl(\alpha,\frac{(1+\epsilon)}{\epsilon}\kappa h\Bigr)\cap \Xi.
$$
It follows that we can
stably represent any functional in $V_{\alpha}'$
by an element of $\ell_1(\Xi_{\alpha,\epsilon})$, 
as the next lemma demonstrates.
%
%
\begin{lemma}
\label{P_LPR}
Let $\M$ be a Riemannian manifold and
let $\varPi$ and $\Omega\subset \M$ satisfy
Assumptions \ref{A_function_space} 
and \ref{A_domain}, respectively. 
For $\alpha\in\Omega$, 
if $\lambda\in V_{\alpha}'$,
and
 $
 \Xi_{\alpha,\epsilon}
 \subset \Omega$
 with
  $h:=h_{
  \Xi_{\alpha,\epsilon}
  ,\Omega} \le h_0$,
  there is a map 
 $\aaa_{\lambda}:
 \Xi_{\alpha,\epsilon}
 \to \R$
 which satisfies
  the conditions
  \begin{itemize}
  \item $\mathrm{supp}(\aaa_\lambda)\subset \Xi_{\alpha,\epsilon}$ 
  \item  $\sum_{\xi\in
  \Xi_{\alpha,\epsilon}
  } |\aaa_\lambda(\xi)|\le (1-\epsilon)^{-1}\|\lambda\|_{V_{\alpha}'}$,
\item   $\sum_{\xi\in
\Xi_{\alpha,\epsilon}
} \aaa_\lambda(\xi) p(\xi) =\lambda p$
 for all 
 $
 p\in 
 V_{\alpha}
 =
 \varPi|_{(B(\alpha,\kappa h)\cap \Omega)}
 $.
\end{itemize}
 \end{lemma}

 \begin{proof}
 The inequality
  $\max_{\xi\in 
  \revision{\Xi}_{\alpha,\epsilon}
  }|p(\xi)|
 \ge 
(1-\epsilon)
 {\|p\|_{V_{\alpha}}}
 \revisioncom{Referee 2: Minor comment 6}$
 holds for all $p\in V_{\alpha}$ 
 by Corollary 
 \ref{C_norming}.
  In other words, the sampling operator
   $$
 S: V_{\alpha}
 \to 
 W_{\alpha,\epsilon}
:p
\mapsto 
p|_{
\Xi_{\alpha,\epsilon}
}
 $$ 
 is bounded below 
 by $1-\epsilon$.
Let
 $\mathsf{R}_S\subset
  W_{\alpha,\epsilon}
  $ denote the range of $S$.
 By restricting to $\mathsf{R}_S$, we have 
 $\|S^{-1}\|_{R_S\to V_{\alpha}} \le (1-\epsilon)^{-1}$, and
 so  the dual map
  $(S^{-1})':V_{\alpha}'\to (\mathsf{R}_S)'$ is bounded
 by $\|(S^{-1})'\|_{V_{\alpha}'\to (\mathsf{R}_S)'}\le (1-\epsilon)^{-1}$.
 Thus, the linear functional 
$(S^{-1})'
\lambda
= 
\lambda  \circ S^{-1}\in (\mathsf{R}_S)'$,  
has norm
$\|(S^{-1})'\lambda \|_{(\mathsf{R}_S)'}
\le 
(1-\epsilon)^{-1}\|\lambda\|_{V_{\alpha}'}$.
Furthermore, $(S^{-1})'
\lambda$
can be  extended to 
 $
 W_{\alpha,\epsilon}
 '$
by preserving its norm.
By identification of 
$\bigl(\ell_{\infty}(
\Xi_{\alpha,\epsilon}
)\bigr)'$ with 
$\ell_1(
\Xi_{\alpha,\epsilon}
)$,
 there is  
 $v\in\R^{
 \Xi_{\alpha,\epsilon}
 }
  $ 
 with 
$\|v\|_{\ell_1(
\Xi_{\alpha,\epsilon}
)}\le (1-\epsilon)^{-1} 
\|\lambda\|_{V_{\alpha}'}
$
for which 
$$
p(\alpha) = \delta_\alpha p  =
 \bigl((S^{-1})'\delta_\alpha \bigr)(Sp)= 
 \sum_{\xi\in 
 \Xi_{\alpha,\epsilon}
 } 
 v_{\xi} p(\xi)$$
 for all 
 $p\in 
 V_{\alpha}
 $. Letting 
 $$\aaa_{\lambda}(\xi):=
 \begin{cases} v_{\xi}&
 \text{for } \xi \in 
 \Xi_{\alpha,\epsilon}
 \\
 0&\text{otherwise}
 \end{cases}
 $$
completes the proof.
 \end{proof}
 
 \subsection{Local reproduction of point evaluation and directional derivatives}\label{SS_local_algebraic}
For $\alpha\in \Omega$, the functional $\delta_{\alpha}:p\mapsto p(\alpha)$
satisfies 
$|\delta_{\alpha}p|\le \sup_{z\in B(\alpha,\kappa h)\cap \Omega}|p(z)|$, or equivalently $\|\delta_{\alpha}\|_{V_{\alpha}'}\le 1$.
Thus there is $\aaa^0: \Xi \times \Omega \to \R$ for which
$$\sum_{\xi\in \Xi }{\aaa^0}(\xi,\alpha) p(\xi) =p(\alpha) $$
for all $p\in \varPi$ and $\alpha \in\Omega$.

Similarly, by Assumption \ref{A_function_space}, 
for a  tangent vector 
$v\in
T_0^1\M_{\alpha}
$,
the functional
$$\lambda_v:p\mapsto D_vp(\alpha) = \nabla p(\alpha) v$$ has norm 
$\|\lambda_v\|_{V_{\alpha}'}\le \frac{\markov}{\kappa h}\|v\|_{T_0^1\M_{\alpha}}$ by
(\ref{F_LM}).
Thus, we can define 
${\aaa^1}(\cdot,\cdot): \Xi \times \Omega \to 
T_1^0\M
$
pointwise, 
that means for any $v\in T_{0}^1\M_\alpha$, 
we set  
$ \aaa^1(\xi,\alpha) v
:=  
\aaa_{\lambda_{v}}^1(\xi)$, where $\aaa_{\lambda_{v}}^1$ is guaranteed by
Lemma \ref{P_LPR}. 
Thus, the map  
${\aaa^1}$
has
 the  following reproduction property: for all $p\in \varPi$ and $\alpha\in \Omega$
$$
\sum_{\xi\in \Xi }p(\xi) {\aaa^1}(\xi,\alpha)  =\nabla p(\alpha).
$$
Furthermore, the locality condition  
$\mathrm{supp}\bigl(\aaa^j(\cdot,z)\bigr)\subset B\bigl( z,(1+\epsilon)\kappa h/\epsilon \bigr)$
holds for   $j=0,1$,
as do the stability conditions
\begin{equation}
\label{LPR_stability_bounds}
\sum_{\xi\in \Xi} 
\bigl| {\aaa^0}(\xi,\alpha) \bigr|\le \frac{1}{1-\epsilon}
\quad
\text{ and }
\quad
\sum_{\xi\in \Xi} 
\bigl\| {\aaa^1}(\xi,\alpha) \bigr\|_{ T_1^0\M_\alpha
} \le  \frac{1}{1-\epsilon} \frac{\markov}{\kappa h}.
\end{equation}

We can achieve higher order analogs with a stronger  Markov inequality.
%
%
\begin{assumption}
 \label{A_Higher}
We  assume that $\varPi\subset C^{k}(\M)$
and  there is a $r_*>0$
so that for every $j\le k$, 
there is a constant
 $\markov_j$ 
 such that for every $p\in  \varPi$, 
$0<r<r_*$,  $x_0\in \M$,
the following inequality holds:
%
%
 \begin{equation}
 \label{higher_order_markov}
 \|\nabla^j p\|_{C(\overline{B(x_0,r)})} 
 \le 
 \frac{\markov_j}{r^j} \|p\|_{C(\overline{B(x_0,r)})}.
 \end{equation}
%
%
 \end{assumption}
%
%
If this holds, then for $v_1,\dots,v_k\in T_0^1\M_{\alpha}$,
 the functional $\lambda_{v_1,\dots,v_k}\in \varPi'$ of the form
 $$p\mapsto \lambda_{v_1,\dots,v_k} p= \nabla^kp(\alpha) (v_1,\dots,v_k)$$
 satisfies a bound of the form 
 $|\lambda_{v_1,\dots,v_k}p| 
 \le 
 \frac{\markov_k}{\kappa^k h^k}
 \|p\|_{C(B(\alpha,\kappa h)} \|v_1\|_{T_{0}^1\M_{\alpha} }\dots \|v_k\|_{T_{0}^1\M_{\alpha}}$.
 Setting $C:= \markov_j \kappa^{-j} $,
 Lemma \ref{P_LPR}  guarantees 
 existence of
 $\aaa^{k}:\Xi\times \Omega \to T_k^0\M$
 so that the 
 reproduction
  formula 
  $\sum_{\xi\in\Xi} p(\xi) \aaa^{k}(\xi,\alpha) = \nabla^kp(\alpha)$ holds for all $p\in \varPi$ and $\alpha \in \Omega$,
 along 
 with 
the support condition
  $\mathrm{supp}\bigl(\aaa^k(\cdot,\alpha)\bigr)\subset B(\alpha,(1+\epsilon)\kappa h/\epsilon)$
 and  
 the stability condition
 $
 \sum_{\xi\in \Xi} 
\bigl\| {\aaa^k}(\xi,\alpha) \bigr\|_{ T_k^0\M_\alpha} 
\le 
 \frac{C}{1-\epsilon} h^{-k}.
$

%
%
%
\subsection{Regularity of generalized polynomial reproductions}
\label{SS_regularity}
Suppose now that $\Omega\subset \M$ satisfies Assumption \ref{A_domain},
and $(E,\pi)$ is a real, $J$-dimensional vector bundle over $\M$ 
so that each fiber $E_z=\pi^{-1}(z)$ has norm
$\|\cdot \|_{E_z}$. Suppose further that $\opL:C^{\infty}(\M) \to C^{\infty}(\M;E)$ 
is a linear map taking smooth functions on $\M$ to smooth sections over $E$.

We generalize slightly the above construction, to consider
 $\tilde{a}:\Xi\times \Omega\to E$ 
 which satisfies, for all $z\in \Omega$
\begin{equation}\label{E_general}
\begin{cases} 
\revision{ \mathrm{supp }\ \tilde{a}(\cdot,z)\subset B(z,R),} \\
\revision{  \sum_{\xi\in \Xi}\|\tilde{a}(\xi,z)\|_{E_z}\le K,}\\
\revision{ (\forall p\in \varPi)\  \sum_{\xi\in \Xi}p(\xi) \tilde{a}(\xi,z)  = \opL p(z).}
 \end{cases}\revisioncom{Referee 2: Minor Comment 7}
\end{equation} 
For $\varPi\subset C^{\infty}(\M)$,
the results of the previous section fit this construction with
 $\opL =\nabla^k$,  $E= T_k^0\M$,
$R= (1+\epsilon)\kappa h/\epsilon$ 
and $K = \frac{C}{1-\epsilon} h^{-k}$.

The following proposition, generalized from the scalar, Euclidean result  \cite[Lemma 10]{RW},
 shows that such generalized local polynomial reproductions can be smooth, as well.
 \begin{lemma} 
 \label{L_regularity}
 If $\tilde{a}:\Xi\times \Omega\to E$ satisfies (\ref{E_general})
and if $\varPi\subset C^{\infty}(\M)$, then for any 
$\revision{\varepsilon}>0$\revisioncom{we
have changed $\epsilon$ to $\varepsilon$.},
there is an ${a}:\Xi\times \Omega\to E$
with ${a}(\xi,\cdot)\in C^{\infty}(\M)$ for all $\xi\in\Xi$, along with
\begin{equation*}
\begin{cases} 
\revision{ \mathrm{supp }\ {a}(\cdot,z)\subset B(z,R+\varepsilon),} \\
\revision{ \sum_{\xi\in \Xi}\| {a}(\xi,z)\|_{E_z}\le K+\varepsilon,}\\
\revision{ (\forall p\in \varPi)\  \sum_{\xi\in \Xi}p(\xi) {a}(\xi,z)  = \opL p(z).}
 \end{cases}\revisioncom{Referee 2: Minor Comment 7}
 \end{equation*}
\end{lemma}
\begin{proof}
Let $\{p_j\mid 1\le j\le \polDim\}$ 
be a basis for $\varPi$ (so $\polDim:=\dim \varPi$).
Pick $y\in\Omega$,
and let $\Xi_{y}:= \Xi \cap 
B(y, R)$.
Because 
$\{\delta_{\xi}\mid \xi\in \Xi_{y}\}$
is a norming set and therefore spans $\varPi'$, 
it contains a 
 subset 
$$\{\delta_{\xi_j}\mid 1\le j\le \polDim
\}
\subset \{\delta_{\xi}\mid \xi \in \Xi_{y}\}, $$
which is linearly independent in $(\varPi)'$.
Define complementary point sets  
$$\Xi_y^{\flat}:= \{\xi_1,\dots,\xi_{\polDim}\}\text{  and. }
  \Xi_y^{\sharp}:=\Xi_{y,\revision{\varepsilon}}\setminus \Xi_y^{\flat}.$$

Consider now a neighborhood $U_y$ of $y$ 
with chart $\psi_y:U_y\to V_y\subset \R^d$ satisfying $\psi_y(y)=0$,
and the 
trivialization of $E$ over $U_y$,
$\Psi_y: \pi^{-1}(U_y)\to U_y\times \R^{J}$.
For simplicity, we identify $\Psi_y$ with the  isomorphism it induces on the fibers $\pi^{-1}(z)$.
Sof
for $z\in U_y$, we write
$\Psi_y:\pi^{-1}(z)\to  \R^{J}$
in place of $\Psi_y:\pi^{-1}(z)\to \{z\}\times \R^{\revision{J}}$\revisioncom{Referee 2: Minor comment 8} 
(i.e., we compose $\Psi_y$ with
the map $(z,\vec{v})\mapsto \vec{v}$).

Consider the
function $G_y: V_y\times \R^{ J \times \polDim} \to \R^{J \times \polDim}$
defined, for $j\le \polDim$, by 
$$
\bigl(G_y(v,\vec{B})\bigr)_j :=
\Psi_y\bigl(\opL  p_j( \psi_y^{-1} v) \bigr)- 
\sum_{\xi_k\in \Xi_y^{\flat} } (\vec{B})_k p_j(\xi_k) 
-
\sum_{\zeta\in \Xi_y^{\sharp}}
 \Psi_y\bigl( \aaa(\zeta,y)\bigr)p_j(\zeta),
$$ 
where $(\vec{B})_k\in \R^J$ is the $k$th column of 
$\vec{B}\in \R^{J \times \polDim}
$. 
Note that the last term is constant in both $v$ and $\vec{B}$,
while  the middle term is constant in $v$ and linear in $\vec{B}$.
Namely, $\sum_{\xi_k\in \Xi_y^{\flat} } (\vec{B})_k p_j(\xi_k)$,
is the $j$th column of  $ \vec{B} \vec{P}\in \R^{J \times \polDim}$,
where $\vec{P}= \bigl( p_j(\xi_k)\bigr)_{k,j}\in \R^{\polDim \times \polDim}$
is a Vandermonde-type matrix for $\revision{\varPi}$\revisioncom{fixed a typo} and $\Xi_y^{\flat}$.

For $v\in V_y$, set
$\vec{B}^{\diamond}(v) := \bigl(\vec{A}(v) - \vec{C}\bigr) \vec{P}^{-1}$
with $\vec{A}(v)\in \R^{J \times \polDim}$ and $\vec{C}\in \R^{ J \times \polDim}$ 
defined column-wise as
$$\bigl(\vec{A}(v)\bigr)_j := \Psi_y\bigl(\opL  p_j( \psi_y^{-1} v) \bigr)
\quad
\text{ and }
\quad
\bigl(\vec{C}\bigr)_j 
:= 
\sum_{\zeta\in \Xi_y^{\sharp}}
 \Psi_y\bigl( \aaa(\zeta,y)\bigr)p_j(\zeta).
 $$
It follows that for all $v\in V_y$,  $G_y(v,\vec{B}^{\diamond}(v))=\vec{0}$.
Furthermore, for $v=0$,
the linear independence of
the functionals $ \{\delta_{\xi_j}\mid \xi_j\in \Xi_y^{\flat}\}$ guarantees that
$\vec{B}^{\diamond}(0)  =\left(\Psi_y({\aaa}(\xi_k,y))\right.)_{k\le \polDim}$.

From this, we define $a_y^*:\Xi\times U_y\to E$ 
at the point $z\in U_y$, with $v= \psi_y(z)$, as
$$a_y^*(\xi,z) = 
\begin{cases}
\Psi_z^{-1}\bigl( \vec{B}^{\diamond}(v)\bigr)_k
& 
\xi_k =\xi\in \Xi_y^{\flat} \\
\Psi_z^{-1}\Psi_y{\aaa}(\xi,y)& 
\xi \in \Xi_y^{\sharp}\\
0&\xi \in \Xi\setminus \Xi_{y}.
\end{cases}$$
For every $\xi\in \Xi$, $a_y^*(\xi,\cdot)$ is in $C^{\infty}(U_y)$.

Note that for $v=0$,
$
a_y^*(\xi,y)=
{\aaa}(\xi,y)
$ 
so
$\sum_{\xi\in\Xi} 
 \|a_y^*(\xi_k,y)\|_{E_y}
 \le K$ by assumption.
By continuity,
 there is a neighborhood $\tilde{U}_y\subset U_y$ of $y$ 
with corresponding Euclidean neighborhood 
$\tilde{V}_y=\psi_y(\tilde{U}_y)$ of $0$,
 so that  for all 
 $z\in  \tilde{U}_y$, 
 $\sum_{\xi\in \Xi}
\|{a}_y^*(\xi,z) \|_{E_z} \le 
K+\revision{\varepsilon}.$

By decreasing the neighborhood even more,
so that  $\tilde{U}_y\subset B(y,\revision{\varepsilon})$ holds,
we have for $\xi \in \Xi_{y}$ and $z\in \tilde{U}_y$, that
$$
\dist(z,\xi) 
\le 
\dist(z,y)
+\dist(y,\xi)
\le
R+\revision{\varepsilon}.$$ 
Thus if $\dist(z,\xi)>R+\revision{\varepsilon}$, then $a_y^*(\xi,z)=0$.

By compactness of $\Omega$, there is a finite cover of the form 
$\Omega = \bigcup_{\ell=1}^L \tilde{U}_{y_\ell}$. Denote by
${a}_\ell:\Xi \times \Omega\to \R$ 
the  extension by zero of  
${a}_{y_\ell}:\Xi \times \tilde{U}_{y_\ell}\to E$. 
Let $\bigl(\psi_\ell\bigr)_{\ell=1\dots L}$ 
be a smooth partition of unity subordinate to this cover: i.e.,
consisting of functions 
$\psi_\ell:\Omega\to[0,1]$ with 
$\mathrm{supp}(\psi_\ell)\subset \tilde{U}_{y_\ell}$ and
$\sum_{j=1}^L \psi_\ell =1$.
 Then ${a}:\Xi\times \Omega\to \R$ defined by
 ${a}(\xi,z):=
 \sum_{\ell=1}^L \psi_\ell(z) a_{\ell}(\xi,z)$
 is the desired smooth local polynomial reproduction.
\end{proof}

%
 %
 %
 %
 %
\section{Moving least squares (MLS)}
\label{S_MLS}
In this section, we consider a  counterpart 
to the construction of section \ref{S_Main}. 
It has the advantage that stable, local reproduction of derivatives follows automatically,
without need of a separate construction. 
An extra requirement (for our results, not for the implementability of the method)
 is quasi-uniformity of the point set $\Xi$, which is described below.

We consider
(\ref{eq:mlsprimal}),
 for
$ \varPi$ satisfying Assumption \ref{A_function_space}
 and $\Omega\subset \M$ satisfying \ref{A_domain}.
 Then for $f\in C(\Omega)$,
we define the MLS approximant via the pointwise 
formula
\begin{align}
\label{eq:genmlsprimal}
  \mathcal{M}_{\Xi}f(z):=p_z^*(z)
    \text{ where }
    p_z^*=\arg\min_{p \in \varPi} \sum_{\xi \in \Xi} \left(p(\xi) -f(\xi) \right)^{2} \weightkernel(\xi,z).
\end{align}
Here $z \in \M$ is a given  point and 
$\weightkernel$ is a given weight function.
Note that this is a reformulation of \eqref{eq:mlsprimal} with $\mathcal{P}_m(\R^d)$ replaced by $\varPi$,
and a general data vector $\bigl(y_{\xi}\bigr)_{\xi\in\Xi}$ replaced by
sampled data 
$\bigl(f(\xi)\bigr)_{\xi\in\Xi}$
(in other words, we express
the MLS approximant $\mathcal{M}_{\Xi}$ 
as an operator on $C(\Omega)$, 
rather than one acting on 
data).

\revision{For the remainder of this section, we set
$\centerCard:=\revision{\mathrm{card}}(\Xi)$ 
and $\polDim:=\dim(\varPi)$. 
We enumerate $\Xi=\{\xi_j\mid j\le \centerCard\}$
and provide a basis $\{p_j\mid j\le M\}$ for $\varPi$.}
\revisioncom{We moved this earlier in the section.}
\subsubsection*{Initial assumptions on $\weightkernel$}
To ensure the above problem is well-posed and stable,
we assume $\weightkernel$  
satisfies, for some constants $c_0\in(0,1)$ \& $c_1\in (0,1]$, the following:
\begin{align}
\dist(\zeta,z)\ge\delta& \Longrightarrow \weightkernel(\zeta,z)=0,\label{outer_support}\\
  \dist(\zeta,z)\le c_0 \delta &\Longrightarrow  \weightkernel(\zeta,z)\ge c_1.\label{inner_support}
\end{align}
In order  to ensure that each ball $B(z,c_0 \delta)$ contains enough points from $\Xi$ to stably reproduce $\varPi$ at $z$, we
also assume
\begin{equation}\label{delta_condition}
 \delta \ge \frac{3\kappa}{c_0} h.
\end{equation}

In this case, there is little benefit in allowing a variable $\epsilon$. 
(As described in Remark \ref{R_MLS_stability}, and in contrast to the theoretically constructed 
local $\varPi$ reproduction, 
the stability bounds we present for \eqref{eq:genmlsprimal} cannot be brought arbitrarily close to $1$.)
Thus, we select $\epsilon=1/2$.
By Lemma \ref{P_LPR}, we may take $c_0\delta \ge 3\kappa h$.
Both constants  $c_0$ and ${c_1}$ affect the stability of the scheme, as demonstrated below in Lemma \ref{L_MLS_stability}.

The choice $\weightkernel(\zeta,z) = \Phi(\frac{\dist(z,\zeta)}{\delta})$ for 
a  compactly supported, non-negative, continuous  $\Phi:[0,\infty)\to [0,\infty)$
will easily satisfy (\ref{outer_support}) and (\ref{inner_support}),
although this is not necessary (and {may} not be desirable for some problems).
For embedded manifolds, considered in section \ref{SS_MLS}, we use a weight function depending on the Euclidean distance, i.e., 
$\weightkernel(\zeta,z) = \Phi(\frac{|z-\zeta|}{\delta})$. 

For data which is highly non-uniform, it may be preferable to use a weight function
with support which varies spatially (so that $\weightkernel(\cdot,z)$ has support which changes
with $z$); in such a case, (\ref{outer_support}), (\ref{inner_support}) and (\ref{delta_condition})
could be replaced by suitable hypotheses on 
a map $x\mapsto \delta_x$ which reflects the local distribution of points near $x$.

\subsubsection*{Quasi-uniformity} For a point set $\Xi$, we consider 
the separation radius $$q:=q_{\Xi}:=\min_{\xi\in\Xi} \dist(\xi,\dist(\Xi\setminus\{\xi\})).
\revisioncom{We have removed the
mesh ratio $\rho=h/q$
modified mesh ratio $\tilde{\rho}=\delta/q$.}$$

\subsection{Shape functions\label{sec:shape_functions}}
The Backus-Gilbert MLS approach~\cite{BG} shows that 
\begin{align*}
    \mathcal{M}_{\Xi} f(z)=\sum_{\genfrac{}{}{0pt}{}{\xi \in \Xi}{\weightkernel(\xi,z)>0}} b^{\star}(\xi, z) f(\xi),
\end{align*}
where the {\em shape functions} $b^{\star}(\xi,\cdot): \M\to \R$ can be obtained pointwise via  
\begin{align}
\label{B-G}
    b^{\star}(\cdot,z)&=\arg\min_{c\in \R^{\Xi}} 
    \sum_{\genfrac{}{}{0pt}{}{\xi \in \Xi}{\weightkernel(\xi,z)>0}} \frac{1}{\weightkernel(\xi,z)} \left|c_{\xi}\right|^2\\
    \text{subject to} &\sum_{\genfrac{}{}{0pt}{}{\xi \in \Xi}{\weightkernel(\xi,z)>0}} c_{\xi} p(\xi) = p(z) \text{ for all } p \in \varPi. \nonumber
\end{align} 
\revision{We point out that (\ref{B-G}) and the MLS 
(\ref{eq:genmlsprimal})
are related by 
\begin{align}\label{E_BG}
     \mathcal{M}_{\Xi} f(z)=
     \p(z)(\PP^T \W(z) \PP )^{-1} \PP^T\W(z) f|_{\Xi}
     =
     \sum_{\xi\in\Xi}
     b^{\star}(\xi, z) f(\xi)
\end{align}
where $\PP=\bigl(p_j(\xi_k)\bigr)\in \R^{\centerCard\times M}$
is a full-rank Vandermonde matrix,
 $\p(z)\in \R^{M\times 1}$ is the vector
 with $j$th entry $(\p(z))_j=p_j(z)$,
and
$\W(z)\in \R^{\centerCard\times \centerCard}$ 
is the diagonal  matrix  chosen to have
$j$th diagonal  entry 
$
(\W(z) )_{j,j}$}
\revisioncom{Referee 2, main comment 1}.
We now show that $b^{\star}$ provides a stable, local $\varPi$-reproduction. 
A 
basic combinatorial estimate we need for this result 
is that for any $z\in \Omega$, and $R<r^*$, 
  \begin{equation}
 \label{quasi-uniform}
 \revision{\mathrm{card}}\{\xi \in \Xi\cap B(z,R)\}\le \frac{\mu(B(z,R)}{\min_{\xi} \mu(B(\xi,q))}
 \le
 \frac{\omega_{\Omega}R^d}{\alpha_{\Omega}q^d}
\end{equation}
holds.
Point sets where $q\sim R$, with $R=\delta$ or $h$, are called quasi-uniform \revision{if the constants of equivalence are independent of $\Xi$)}, 
and many estimates can be
expressed in terms of a mesh-ratio $R/q$. 
In this section, we  leave estimates in terms of cardinalities whenever possible, and make no explicit assumption of quasi-uniformity.\revisioncom{Here and throughout
we have replaced $\#()$ by $\mathrm{card}()$}
%
%
%
%
\begin{lemma}
\label{L_MLS_stability}
If $\varPi$ satisfies Assumption \ref{A_function_space},
 $\Omega\subset \M$ satisfies Assumption \ref{A_domain},
 then there is a constant $C$ so that
 $\Xi\subset \Omega$ is sufficiently dense 
 (with  $h\le h_0$)
  and
   if
  $\weightkernel$ satisfies (\ref{outer_support}), (\ref{inner_support})
  and (\ref{delta_condition})
 then for 
$z\in\Omega$,
  \begin{itemize}
   \item for all $p\in\varPi$, 
 $$p(z) = \sum_{\xi\in\Xi} b^{\star}(\xi,z)p(\xi),$$
  \item $\mathrm{supp} \bigl(b^{\star}(\cdot,z)\bigr) 
  \subset 
  B(z,  \delta)\cap \Omega$,
\item 
$\sum_{\xi\in\Xi} |b^{\star}(\xi,z)|
\le 
\revision{
\frac{2}{\sqrt{c_1}}
\Bigl(
\revision{\mathrm{card}}\bigl(\Xi\cap B(z,\delta) \bigr)
\Bigr)^{1/2}}
\le
2
\sqrt{\frac{\omega_{\Omega}}{c_1\alpha_{\Omega}}}
\left( \frac{\delta}{q}\right)^{d/2}
$.
 \end{itemize}
\end{lemma}
\begin{proof}
From (\ref{B-G}) it is clear that $b^{\star}$ reproduces point evaluations. Because $\weightkernel$ satisfies (\ref{outer_support}),
it follows that if $\dist(\xi,z)\ge\delta$, then $b^{\star}(\xi,z)=0$.

To handle the third item, we apply Cauchy-Schwarz inequality, obtaining
$$\sum_{\xi \in \Xi}| b^{\star}(\xi,z) |\le
\sqrt{ \sum_{\genfrac{}{}{0pt}{}{\xi \in \Xi}{\weightkernel(\xi,z)>0}}{\weightkernel(\xi,z)}}
\sqrt{\sum_{\genfrac{}{}{0pt}{}{\xi \in \Xi}{\weightkernel(\xi,z)>0}} \frac{1}{\weightkernel(\xi,z)} 
\left|  b^{\star}(\xi,z)\right|^2}.
$$
Boundedness of the weight $\weightkernel$ and quasi-uniformity
of $\Xi$, via (\ref{quasi-uniform}), guarantee that 
$$
\left( \sum_{\genfrac{}{}{0pt}{}{\xi \in \Xi}{\weightkernel(\xi,z)>0}}{\weightkernel(\xi,z)}\right) 
\le 
\revision{\mathrm{card}}\Bigl\{\xi \in \Xi\mid \weightkernel(\xi,z)>0\Bigr\}
\le
 \frac{\omega_{\Omega}}{\alpha_{\Omega}}\left( \frac{\delta }{q}\right)^d   .
 $$
 To handle the second factor, we use the local $\varPi$-reproduction $\tilde{a}^0$ of section \ref{SS_local_algebraic},
 noting that  $\sum_{\xi \in \Xi}  \left| \tilde{a}^0(\xi,z)\right|^2<(\sum_{\xi \in \Xi} \left| \tilde{a}^0(\xi,z)\right|)^2\le (1-\epsilon)^{-2}$.
From (\ref{LPR_stability_bounds}), with $\epsilon=\frac12$, we have
$$
\left(\sum_{\genfrac{}{}{0pt}{}{\xi \in \Xi}{\weightkernel(\xi,z)>0}} \frac{1}{\weightkernel(\xi,\zeta)} \left| b^{\star}(\xi,z)\right|^2\right)
\le 
\left(\sum_{\genfrac{}{}{0pt}{}{\xi \in \Xi}{\dist(\xi,z)<c_0 \delta}} \frac{1}{\weightkernel(\xi,z)} \left| \tilde{a}^0(\xi,z)\right|^2\right)
\le  \frac{4}{c_1} .$$
\end{proof}
\begin{remark}
\label{R_MLS_stability}
Although parameters $ c_0<1 $ and  $\epsilon<1$ can be permitted to get arbitrarily close to $1$ (and $c_1$ may equal 1),
by following the above argument, the best stability estimate $C$ in 
$\sum_{\xi\in\Xi} |b^{\star}(\xi,z)|\le C$ 
will be no smaller than
$2\sqrt{\frac{\omega_{\Omega}}{c_1 \alpha_{\Omega}}}
\kappa^{d/2} \rho^{d/2}$,
where $\rho = h/q$.
This is  in contrast to 
$(1-\epsilon)^{-1}$ from (\ref{LPR_stability_bounds}), which can get arbitrarily close to $1$.
\end{remark}

\subsection{Construction and smoothness}
The solution to the variational problem (\ref{B-G}) is given by linear projection. 
Viewing $b^*(\cdot,z)$ as a column vector in 
$\R^{\centerCard\times 1}$
we have
%
%
\begin{equation}
\label{BG-formula}
b^*(\cdot,z) = \W(z) \PP(\PP^T\W(z) \PP)^{-1} \p(z).
\end{equation}
%
%
This follows from (\ref{E_BG}).
Please note, that condition \eqref{delta_condition} ensures that $\Xi \cap B(z,c_0\delta)$ contains a unisolvent set of centers. Thus $\PP$ is a full rank matrix, which makes the inverse $(\PP^T\W(z) \PP)^{-1}$ well-defined. 
Consequently, $b^*$ is in $C^k(\M)$ if $\varPi\subset C^k(\M)$ and each function $\weightkernel(\xi,\cdot)$ is as well.
Indeed, by linearity we have, for $z\in \mathrm{int}(\Omega)$, that
$$\nabla^k p(z) = \sum_{\xi\in \Xi} p(\xi) \nabla^k b^*(\xi,z)$$
(with covariant derivative
applied to the second entry).
It follows that 
$$\nabla^k b^*:\Xi\times \Omega\to T^{0}_{k}\M$$ 
reproduces $\nabla^k$ on $\varPi$ 
and is local  in the sense that
$\mathrm{supp}\nabla^k b^{\star}(\cdot,z) \subset B(z,\frac{ 3\kappa }{ c_0}h h)$.

  The following assumption is sufficient to get regularity bounds for the MLS shape functions,
  while allowing some flexibility in choosing the weight function.
 In addition to implying (\ref{outer_support}) and (\ref{inner_support}), 
  it allows weights which depend on
 different distances  (see Remarks  \ref{R_riemannian_metric} and \ref{R_euclidean_metric}).
  On its face, it may even allow for spatial variation of the weight function, although this is not
  considered in the present article.
  
%
%
%
\begin{assumption}
\label{A_weight}
Assume there exist 
constants
 $0< \vartheta_1\le  \vartheta_2<\infty $ and 
 $C_k$, for $k\in\N$, and a collection of functions 
 $\{F_{\xi}\in C^k(\Omega\setminus \{\xi\})\mid \xi\in \Omega\}$, 
so that for all $z\in B(\xi,r^*)\setminus \{\xi\}$,
%
 \begin{equation}
 \label{eq:homogeneity}
  \vartheta_1 \dist(\xi,z)\le F_{\xi}(z) \le  \vartheta_2 \dist(\xi,z)
  \quad 
  \text{ and }
  \quad
\|\nabla^k F_{\xi} (z)\|_{T_k^{0}\M_z}
\le C_k (\dist(z,\xi))^{1-k}
 \end{equation}
%
 and 
 the weight function satisfies, for $\xi \in \Xi$,
%
$$
 \weightkernel(\xi,z) 
 = 
 \Phi\Bigl(\frac{F_{\xi}(z)}{\delta}\Bigr).
 $$
%
 We further assume that 
 $\Phi\in C^{\infty}([0,\infty))$, has support in the interval $[0,1/ \vartheta_1]$, and satisfies the conditions 
$\sup_{z\in [0, \vartheta_1^{-1}]}|\Phi(z)|=1$ and $\Phi(z) =1$ for $|z|\le 1/2$.
\end{assumption}
%
%
%

The most natural weight in this context uses the Riemannian distance:
%
%
\begin{remark}
\label{R_riemannian_metric}
If $F_{\xi}(z) = \dist(\xi,z)$, then Assumption \ref{A_weight} is satisfied,
 since the metric equivalence (\ref{eq:expmap})
implies that  
$\|\nabla^k F_{\xi} (z)\|_{T_k^{0}\M_z} \le C _2\|F_{\xi}\circ\Exp_{\xi} \|_{C^k(A)}$,
where we employ the annulus $A=\{x \mid  \frac12 \dist(z,\xi)\le |x|\le 2\dist(z,\xi)\}$.
In particular, (\ref{eq:homogeneity}) follows by differentiating the homogeneous function
 $x\mapsto |F_{\xi}(\Exp_{\xi}(x))| = |x|$.
\end{remark}
%
%

A weight function satisfying  this assumption   satisfies  
 the hypotheses of Lemma \ref{L_MLS_stability}. 
%
%
%
\begin{lemma}
If  
Assumption \ref{A_weight} holds,
then 
so do
(\ref{outer_support}) and (\ref{inner_support})
with $c_0=\frac{1}{2 \vartheta_2}$ and $c_1=1$. 
\end{lemma}
%
%
We may thus update our hypothesis for $\delta$: since
$
c_0 =\frac{1}{2 \vartheta_2}$
we require 
$
\delta\ge \frac{ 3\kappa }{ c_0}h= 6  \vartheta_2 \kappa h.
$
\begin{proof}
If $\dist(\xi, z)>\delta$, then $F_{\xi}(z)/\delta> \frac1{ \vartheta_1}$, so $\weightkernel(\xi,z)=0$.
Similarly, if $\dist(\xi,z)\le \frac{1}{2 \vartheta_2}\delta$, then  $F_{\xi}(z)/\delta\le \frac1{2}$, so $\weightkernel(\xi,z)=1$.   
\end{proof}

%
%
\begin{lemma}
\label{L_weight}
If $\Xi\subset \Omega$ and  $\weightkernel$ satisfies 
Assumption \ref{A_weight} 
then there is $C$ so that
$$
\| \weightkernel(\xi,\cdot)\|_{C^k(\Omega)}
\le C \delta^{-k}.
$$
\end{lemma}
%
%
\begin{proof}
We can estimate
the above norm by considering coordinates about $x_0\in \M$. 
Employing (\ref{eq:expmap}), if $z\in B(x_0,r_{\Omega})$, 
we may estimate
$ \|\nabla^k \weightkernel(\xi_j,z)\|_{T_k^{0}\M_z}$
by considering quantities
$\I(x):=D^{\alpha} \Phi\Bigl({F_{\xi_j}(\Exp_{x_0}(x))}/{\delta}\Bigr)$
for  $|\alpha|\le k$ and $x\in B(0,r_{\Omega})$.
Applying the chain rule  and product rule shows that $\revision{\I}(x)$\revisioncom{Referee 2: Minor Comment 9} consists of a linear
combination of functions
 of the form
$$
\delta^{-|\ell|} (\Phi^{(\ell)})\Bigl({F_{\xi_j}(\Exp_{x_0}(x))}/{\delta}\Bigr)\ \times \ 
\prod_{j=1}^{\ell}
 D^{\gamma_j} F_{\xi_j}(\Exp_{x_0}(x)) 
 $$
 where $\ell \le |\alpha|$,
 each $|\gamma_j|>0$
 and  $\sum_{j=1}^{\ell} \gamma_j= \alpha $.
 
By the assumptions on $\Phi$, 
$\sup_{z\in [0, \vartheta_1^{-1}]}|\Phi^{(\ell)}(z)|$ is bounded
for every $\ell$.
Furthermore,
for $z$ 
which satisfy
$\dist(\xi,z)\ge \delta/(2 \vartheta_2)$, 
the inequality  $ |\nabla^k F_{\xi} (z)|\le C  \delta^{1-k}$
holds,
with constant $C=(2 \vartheta_2)^{k-1}C_k$.
On the other hand, if $\dist(\xi,z)< \delta/(2 \vartheta_2)$, 
then $F_{\xi}(z)/\delta <\frac12$,
and $D^{\beta} \Phi\Bigl(\frac{F_{\xi_j}(z)}{\delta}\Bigr)=0$.
Thus $\I(x)$ is controlled   by
%
%
$$
\delta^{-|\ell|} 
\left|
(D^{\ell} \Phi)\Bigl({F_{\xi_j}(\Exp_{x_0}(x))}/{\delta}\Bigr)
\right|
\prod_{j=1}^{\ell}
\left| D^{\gamma_j} F_{\xi_j}(\Exp_{x_0}(x)) \right| \le C\delta^{-\ell} \prod_{j=1}^{\ell} \delta^{1-|\gamma_j|} 
= 
C\delta^{-|\alpha|}.
$$
%
%
\end{proof}

\begin{proposition}
\label{P_MLS}
Given a  smooth Riemannian manifold $\M$ with $\varPi\subset C^k(\M)$
for which Assumptions \ref{A_function_space} and \ref{A_Higher} hold,
 a  subset $\Omega\subset \M$ satisfying Assumption \ref{A_domain},
 a weight $\weightkernel$ satisfying Assumption \ref{A_weight},
  there is a constant $C$ 
so that if    $\Xi\subset \M$ and $\delta $ 
satisfy  
$6  \vartheta_2\kappa  h<\delta<r^*$
then for all 
 $z\in \Omega$  and 
 for any $j\le k$ we have
 $$
 \sum_{\xi\in\Xi} \|\nabla^j b^{\star}(\xi,z)\|_{T_{j}^0\M_z} \le C 
 \delta^{-j}  \revision{\mathrm{card}}\bigl(\Xi \cap B(z,\delta)\bigr).
 \revisioncom{Made this estimate depend on cardinality, rather
  than $\tilde{\rho}$ 
  (and fixed a small error).}
 $$
\end{proposition}
%
\begin{proof}
To get a uniform bound over $\Omega$ we cover $\Omega$ by a finite number of neighborhoods of the form $B(\upsilon,r_{\Omega})$
and use normal coordinates in each ball. 
To simplify notation, 
 we suppress the map $\Exp_{\upsilon}$, thus
 $f(z) $ denotes $f(\Exp_{\upsilon}(z))$ throughout this proof.
 
 \revision{The estimate
$\|\nabla^k b^{\star}(\xi,\zeta)\|_{T_k^{0}\M_{\zeta}} 
 \le \|  b^{\star}(\xi,\cdot)\|_{C^k(\Exp_{\zeta}(B(0,r_{\Omega}))} 
$
holds for any  $\xi$ and any $\zeta$.
Thus,
(\ref{eq:expmap})
guarantees  that
 $$
 \|\nabla^k b^{\star}(\xi,\zeta)\|_{T_k^{0}\M_{\zeta}} 
\le 
C_2\|b^{\star}(\xi,\cdot) \|_{C^k(B(0,r_{\Omega}))} 
\le
C_2 \max_{\alpha\le k}
\max_{z\in B(0,r_{\Omega})}
 \sum_{\xi\in\Xi} 
 | D^{\alpha} b^{\star}(\xi,z)|
 $$
where we recall the definition 
$\|u\|_{C^k(B(0,r_{\Omega}))} :=  
\max_{|\alpha|\le k}\max_{z\in B(0,r_{\Omega}))}  |D^{\alpha} u(z)|$.}\revisioncom{Referee 2: minor comment 10}

Let $\G(z) := (\PP^T\W(z) \PP)^{-1}$, 
 and note that 
 $$D^{e_j} \G(z) = -\G(z)(\PP^T D^{e_j}\W(z) \PP)\G(z).$$
 Using (\ref{BG-formula})
 it is a simple
 exercise to show
 that the vector  $\bigl(D^{\alpha} b^{\star}(\xi,z) \bigr)_{\xi\in\Xi}$
 is a linear combination of expressions of the
 form
%
%
 \begin{equation}
 \label{eq:Leibniz}
 D^{\alpha_1} \W(z) \PP
 \left[\prod_{\ell=1}^\mu \left(\G(z) (\PP^T D^{\beta_\ell}\W(z) \PP)  \right)\right]
 \G(z)
 D^{\alpha_2}\p(z)
 \end{equation}
%
%
 with 
 multi-indices satisfying $\alpha_1 + (\sum_{\ell=1}^\mu \beta_{\ell}) + \alpha_2= \alpha$.
 
 We note that the above holds for any basis $\{p_j\mid j\le M\}$.
 We make local selections as follows:
 cover $ B(\upsilon,r_{\Omega}) \cap \Omega$ with a finite number of sets of the form $B(z_\nu,\delta)$ (the 
 number of elements in the cover is unimportant). 
 In each neighborhood $B(z_\nu,\delta)$
 we 
 select  a  basis $\mathfrak{P}^{z_{\nu}}:=\{p_j^{z_\nu}\mid j\le M\}$ 
 for $\varPi$ 
 which is
 orthogonal with respect to the  $L_2(B(z_{\nu},2\delta)\cap \Omega)$ inner product,
 and use this to estimate (\ref{eq:Leibniz}).

For this choice of basis, the Vandermonde matrix $\PP=(p_j^{z_\nu}(\xi_k))_{j,k}$ 
is constant throughout
$B(z_{\nu},\delta)$
and
$\G(z):=\PP^T\W(z) \PP$ and $\p(z)=(p_j^{z_\nu}(z))_{j\le M}$ 
depend only on $z$  within  
$B(z_{\nu},\delta)$.
The matrices $\PP$ and $\G$ are investigated, for this choice of basis, in section \ref{SS_MB}.

By the triangle inequality, it suffices to control each  term of the form (\ref{eq:Leibniz})
with the $\ell_1(\Xi)$ norm,
which in turn can be estimated by 
%
%
\begin{multline*}
\J:=
\Bigl\|
  D^{\alpha_1} \W(z) \PP
 \left[\prod_{\ell=1}^\mu \left(\G(z) (\PP^T D^{\beta_\ell}\W(z) \PP)  \right)\right]
 \G(z)
 D^{\alpha_2}\p(z)
\Bigr\|_{\ell_1(\Xi)} \\
\le
\|D^{\alpha_1} \W(z) \|_{\ell_1(B(z,\delta)\cap\Xi)\to \ell_1(\Xi)}
  \|\PP\|_{\ell_2(M)\to \ell_1(B(z,\delta)\cap\Xi)}\\
  \times
\prod_{\ell=1}^\mu
   \| D^{\beta_\ell}\W(z) \|_{\ell_{\infty}(B(z,\delta)\cap\Xi)\to \ell_{\infty}(B(z,\delta)\cap\Xi)}
   \\
\times 
  \bigl(\|\PP\|_{\ell_2(M)\to \ell_{\infty}(B(z,\delta)\cap\Xi)}\bigr)^{2\mu} 
  \bigl(\| \G(z)\|_{\ell_2(M)\to \ell_2(M)}\bigr)^{\mu+1}
  \|D^{\alpha_2}\p(z)\|_{\ell_2(M)}.
\end{multline*}
%
%
Lemma \ref{L_weight} allows us to make the estimate
 $ \|D^{\gamma} \W(z) \|_{\ell_1(B(z,\delta)\cap\Xi)\to \ell_1(\Xi)}\le C \delta^{-|\gamma|}$.
 By Lemma  \ref{L_pol_col_ vector}, we have
$  \|D^{\gamma}\p(z)\|_{\ell_2(M)}\le C \delta^{-|\gamma|-d/2}$.
By  Lemma \ref{vandermonde}, we have the estimates
$\|\PP\|_{\ell_2(M)\to \ell_1(B(z,\delta)\cap\Xi)} \le C
\delta^{-d/2}
\revision{\mathrm{card}}\bigl(\Xi \cap B(z,\delta)\bigr)$
and
$\|\PP\|_{\ell_2(M)\to \ell_{\infty}(B(z,\delta)\cap\Xi)} \le C\delta^{-d/2}$.
Lemma  \ref{Gram}
gives
$\| \G(z)\|_{\ell_2(M)\to \ell_2(M)}
\revision{\le}\revisioncom{missing $\le$ sign}
C \delta^{d}$.
By combining constants,
we have
\revisioncom{Replaced
 $\tilde{\rho}$ by $\mathrm{card}()$.}
\revision{
\begin{eqnarray*}
\J
 &\le&
  C \delta^{-|\alpha_1|}
  \Bigl(
\delta^{-d/2}
\revision{\mathrm{card}}\bigl(\Xi \cap B(z,\delta)\bigr)
\Bigr) 
 \delta^{ -\sum|\beta_j|}
 \Bigl(\delta^{-d/2}\Bigr)^{2\mu}
\delta^{(\mu+1)d}
 \delta^{-|\alpha_2|-d/2}\\
& =&
 C \delta^{-|\alpha|}
 \revision{\mathrm{card}}\bigl(\Xi \cap B(z,\delta)\bigr)
 \end{eqnarray*}}
and the result follows.
\end{proof}
 
%
%
%
\subsection{Matrix bounds}
\label{SS_MB}
Although we may consider the basis $\{p_j\mid j\le M\}$
to be independent of $z$, we may wish to make different local choices of this basis,
both to aid in analysis and implementation. 
To this end, 
for any  point $x_0\in \Omega$
consider a
basis for $\varPi$
$$\mathfrak{P}^{x_0}:=\{p_j\mid j\le \polDim\},$$
which is orthonormal with respect to $L_2(B(x_0,2\delta)\cap \Omega)$.
The following lemmas estimate norms of vectors and matrices associated with this choice of basis.

%
%
%
\begin{lemma}
\label{L_pol_col_ vector}
There is a constant $C$ so that if 
$\Xi\subset \Omega$ is sufficiently dense,
then
$$\|p_j \|_{C^k( \overline{ B(x_0, \delta)}\cap\Omega)}
\le C\delta^{ -k-d/2} $$
holds.
\end{lemma}
%
%
%
\begin{proof}
\revision{Assumption \ref{A_Higher} followed by Lemma \ref{int_ext_sup}
and Corollary \ref{C_Nikolskii}
gives 
$$\|p_j \|_{C^k( \overline{ B(x_0, \delta)}\cap\Omega)}
\le C \delta^{-k}\|p_j\|_{C( \overline{ B(x_0, \delta)}\cap\Omega)}
\le  C \delta^{ -k-d/2} \|p_j\|_{L_2(B(x_0, \delta)\cap\Omega)}
.$$
The result then follows 
from: $
 \|p_j\|_{L_2(B(x_0, \delta)\cap\Omega)}
 \le\|p_j\|_{L_2(B(x_0, 2\delta)\cap\Omega)}=1$.}
\end{proof}

%
%
%
\begin{lemma}
\label{vandermonde}
There is a constant $C$ so that if 
$\Xi\subset \Omega$ is sufficiently dense and has mesh ratio $\rho= h/q$,
then
for any
$z\in B(x_0,\delta)$,
the
Vandermonde matrix 
$\PP=(p_j(\xi_k))_{j,k}$
constructed by
using points
$\{\xi_k \mid k\le M\}= \Xi \cap B(z,\delta)$
satisfies
 the inequality
$$
\|\PP\|_{\ell_2(\polDim)\to \ell_\tau(B(z,\delta)\cap\Xi)}
\le 
C 
 \delta^{-d/2}
 \bigl(\revision{\mathrm{card}}\bigl(\Xi \cap B(z,\delta)\bigr)\bigr)^{1/\tau},
$$
for $1\le \tau\le \infty$ (with the usual modification $d/\tau=0$ when $\tau=\infty$).
\end{lemma}
%
%
%
\begin{proof}
Let $J=\revision{\mathrm{card}}\bigl(\Xi \cap B(z,\delta)\bigr)$.
For a coefficient 
vector $a\in \R^{\polDim}$, we consider 
the function $p=\sum_{j=1}^\polDim a_j p_j$.
\revision{We have}\revisioncom{Removed $\rho,\tilde
{\rho}$}
$$ 
\|\PP a\|_{\ell_\tau( B(z,\delta)\cap\Xi)}
\le 
C 
J^{1/\tau}
\max_{\dist(\xi_j,z)<\delta}| p(\xi_j)|
\le 
C
J^{1/\tau}
 \|p \|_{C( \overline{B(z,\delta)}\cap\Omega)}.
$$ 
The last norm can \revision{be}\revisioncom{Referee 2: Minor Comment 11} bounded, via Corollary \ref{C_Nikolskii}, as 
$$ 
\|p \|_{C( \overline{B(z,\delta)}\cap\Omega)}
\le 
\mathfrak{N} \delta^{-d/2} 
\|p\|_{L_2(B(z,\delta)\cap\Omega)}.
$$
Because $B(z,\delta)\subset B(x_0,2\delta)$, we have that 
 $ \|p\|_{L_2(B(z,\delta)\cap\Omega)}\le  \|p\|_{L_2(B(x_0,2\delta)\cap\Omega)}$
and the lemma follows by the orthonormality of the basis.
\end{proof}

The next lemma shows that we can control the  matrix norm of 
 $\G(z)^{-1}$.
%
%
%
\begin{lemma}
\label{Gram}
There is a constant $C$ so that if 
$\Xi\subset \Omega$ is sufficiently dense,
$z\in B(x_0,\delta)$
and
$\PP=(p_j(\xi_k))_{j,k}$ 
is
 the
Vandermonde matrix for 
$\mathfrak{P}^{x_0}$
sampled at 
$\Xi \cap B(z,\delta)$
 then 
$$
\|\G(z)^{-1}\|_{\ell_2(M)\to \ell_2(M)}
\le C \delta^{d} .
$$
\end{lemma}
%
%
%
\begin{proof}
Let $a\in \R^{\polDim}$  be coefficients for the polynomial   $p:= \sum_{j=1}^M a_jp_j$.
 Consider the quadratic form
 $Q(a):=a^T \G(z) a=a^T \PP^T\W(z)\PP a= 
 \sum_{\xi_j\in\Xi} |p(\xi) |^2
 \Phi(\frac{F_{\xi_j}(z)}{\delta})$.
 By properties of $\Phi$
 (namely that $\Phi=1$ on $[0,1/2]$
 and $F_{\xi_j}$ takes $B(\xi_j,\delta/(2 \vartheta_2))$ to
 $[0,1/2]$), 
 we have
 $$
 Q(a)
 \ge 
 \sum_{\dist(\xi_j,z)<\delta/(2 \vartheta_2)}|p(\xi) |^2
 \ge  
 \max_{\dist(\xi_j,z)<\delta/(2 \vartheta_2)}|p(\xi_j)|^2
 \ge 
 \frac{1}{4}\|p\|_{C( \overline{B(z,\delta/(2 \vartheta_2))}\cap \Omega))}^2,$$
 where 
 in the last inequality we have 
 used the fact that $\delta/(2 \vartheta_2) \ge 3\kappa h$ to apply the norming set inequality Corollary \ref{C_norming} with $\epsilon=1/2$. 
 
Applying Lemma \ref{int_ext_sup} followed by the doubling property a number of times, gives
%
\begin{eqnarray*}
\|p\|_{C(B(z,\delta/(2\vartheta_2))\cap \Omega)}^2
 &\ge&
2^{-2\doub}  
\left( \frac{2+\sin \omega}{\Gamma_1\sin\omega}\right)^{-2\doub}
 \|p\|_{C( \overline{B(z,\delta/(2 \vartheta_2))})}^2\\
&\ge &
 2^{-2\doub} 
 \vartheta_2^{-\doub}\left( \frac{2+\sin \omega}{\Gamma_1\sin\omega}\right)^{-2\doub}
 \|p\|_{C( \overline{B(x_0,2\delta)}\cap \Omega))}^2 .
 \end{eqnarray*}
%
 By a standard inequality
 we have 
 $
\|p\|_{L_2(B(x_0,2\delta)\cap \Omega)}^2
\le
\mathrm{vol}(B(x_0,2\delta)\cap \Omega)\|p\|_{C( \overline{B(x_0,2 \vartheta_2)}\cap \Omega)}^2
$.
By (\ref{ball_bound}), 
$ \mathrm{vol}(B(z,2 \delta)\cap \Omega)\le \omega_{\Omega}2^d \delta^d$, 
so 
$
\|p\|_{C( \overline{B(x_0,2\delta)}\cap \Omega)}^2 
\ge 
\frac1{\omega_{\Omega}}2^{-d}\delta^{-d} \|p\|_{L_2(B(x_0,2\delta)\cap \Omega)}^2
$.
Orthogonality of the basis guarantees
$\|p\|_{L_2(B(x_0,2\delta)\cap \Omega)}= \|a\|_{\ell_2(\polDim)}$,
so
$$a^T \PP^T\W(z)\PP a\ge C \delta^{-d} \|a\|^2_{\ell_2(M)}$$
and the result follows.
 \end{proof}

\begin{remark}
\revision{
The  norm of  the Gram matrix satisfies $$\|\G(z)\|_{\ell_2(M)\to \ell_2(M)} \le
\|\PP\|_{\ell_2(M)\to \ell_{2} (B(z,\delta)\cap\Xi)}^2 
\|\W(z)\|_{\ell_{2} (B(z,\delta)\cap\Xi)\to\ell_{2} (B(z,\delta)\cap\Xi)} .$$
From Lemma \ref{vandermonde},
$\|\PP\|_{\ell_2(M)\to \ell_{2} (B(z,\delta)\cap\Xi)}^2  \le C  \delta^{-d}  \revision{\mathrm{card}} \bigl(\Xi\cap B(z,\delta)\bigr) $,
while $\Phi(z)\le 1$ implies that
$\|\W(z)\|_{\ell_{2} (B(z,\delta)\cap\Xi)\to\ell_{2} (B(z,\delta)\cap\Xi)} \le 1$.
Consequently,  for $z\in B(x_0,\delta)$,
the relative condition number satisfies  $$\kappa(\G(z))\le C \,\mathrm{card}\bigl(\Xi\cap B(z,\delta)\bigr) $$ for some constant 
$C$ independent of $\delta$ and $\Xi$.
Combining (\ref{delta_condition}) and (\ref{quasi-uniform}) shows that relative condition is bounded independently of $\delta$ and $\Xi$ for quasi-uniform sets of points.
}\revisioncom{We have fixed an error
made in computing the 
matrix norm of $\G(z)$.}
\end{remark}

%
%
%
\section{Application to algebraic manifolds and mesh-free approximation}
\label{S_Alg}
We now apply the results of the previous sections to the case of smooth algebraic varieties.
Specifically, we require that $ X_{0}\subset \R^N$ is the joint zero set of 
polynomials $$P_1,\dots P_{{\D}-\d}\in \mathcal{P}(\R^{\D})$$
and $\M\subset X_{0}$ is the  subset of regular points.
In other words,  the Jacobian
$(\frac{\partial P_j}{\partial x_k})$
has constant (full) rank $\D-\d$ on $\M$.
It follows that $\M$ is an embedded Riemannian submanifold of $\R^N$
and for any compact
$\Omega\subset \M$,
there exists a constant $C_{\Omega}$ so that for any $x,y\in \Omega$,
\begin{equation}
\label{eq:Riemannian-Euclidean}
|x-y|
 \le \dist_{\M}(x,y)\le 
 C_{\Omega}|x-y|
 .
\end{equation}
Furthermore,  the growth estimates and local Markov inequalities of 
\cite{fefferman1996local} and  \cite{bos:etal:1995} apply,
guaranteeing that Assumption \ref{A_function_space} holds on $\M$
for algebraic polynomials $\varPi=\mathcal{P}_{m}(\R^N)$.

In particular, the result 
\cite[Theorem 1.1]{fefferman1996local} 
states that there exists a constant $C_*$ so that 
 for any sufficiently small $r>0$, any $x_0\in\M$ and any
 $m\in\N$,
 the inequality
 %
\begin{align}
\label{Fef_Doubling}
\|  p\|_{C( B_{\M}(x_0,2r))}
&\le
e^{C_* m}
\|  p\|_{C( B_{\M}(x_0,r))}
%
\end{align}
%
holds,  for all $p\in \mathcal{P}_m(\M)$, 
which ensures (\ref{F_Doubling}) from Assumption \ref{A_function_space}.
The first order Markov inequalities given in
\cite[Theorem 1.1]{fefferman1996local} 
and \cite[Main Theorem (2)]{bos:etal:1995}
 also give a first order Markov inequality of the type (\ref{F_LM}) from Assumption \ref{A_function_space}.
 Furthermore, \cite[Proposition 7.1]{bos:etal:1995}
 demonstrates that algebraic manifolds are the only compact, closed Riemannian manifolds on which Markov inequalities
 hold with a constant that depends linearly on the polynomial degree.

Despite the interest in tangential Markov inequalities, the stronger   Assumption (\ref{A_Higher}) does not seem to exist in the literature.  Furthermore, it  does not appear to follow 
 directly by  iterating  the first order inequalities of \cite{fefferman1996local}  or  \cite{bos:etal:1995}.
 We present Lemma \ref{L_Markov}, with proof in the appendix, 
  which follows from the techniques  of \cite{bos:etal:1995}.
  \begin{lemma}\label{L_Markov}
If $\M$ is an algebraic manifold and  $\Omega\subset \M$ is compact, then there is  $\epsilon_*>0$ 
so that for any $m\in \N$ 
and   any $k\in \N$
there is $C$ (depending on $m,k$) so that  if $r<\epsilon_*$,  
$$\|\nabla^k p(y)\|_{C\bigl(\overline{B(x_0,r)}\bigr)}  
\le
C_{m,k}
r^{-k}
\|p\|_{C\bigl(\overline{B(x_0,r)}\bigr)}
$$
holds for all $p\in \mathcal{P}_m(\M)$.
\end{lemma}

Combining Lemma \ref{P_LPR} and Lemma \ref{L_regularity}, we have the following:

%
%
%
\begin{proposition} 
\label{P_algebraic}
Suppose $\M\subset \R^N$ is an algebraic manifold
and
that $\Omega\subset \M$  satisfies  \revision{Assumption}\revisioncom{Referee 1: Minor correction 1} (\ref{A_domain}).
Then for $m\in \N$,
$\Xi\subset \Omega$ with $h= h(\Xi,\Omega)<h_0$  
there is a $C^{\infty}$ smooth, stable, 
local $\mathcal{P}_m$ reproduction. That is, for every $k\in\N$, there
is a constant $C_{k,m}$ and
a map 
${a^k}:
 \Xi \times \Omega \to T_k^0\M$ which satisfy the following four conditions:\revisioncom{Referee 2: Minor Comment 7}
\begin{enumerate}
\item 
for every $z\in \Omega$ if $\dist(\xi,z)>2\frac{(1+\epsilon)}{\epsilon}\kappa h$ then 
${a^k}(\xi,z)=0$
\item 
for every $z\in \Omega$, 
$\sum_{\xi\in \Xi} 
\bigl\| {a^k}(\xi,z) \bigr\|_{ T_k^0\M_z} 
\le 
C_{k,m}
 (1-\epsilon)^{-1} h^{-k}$
\item for every $p\in \mathcal{P}_m(\M)$ and 
$z\in \Omega$, 
$  \sum_{\xi\in \Xi }{a^k}(\xi,z) p(\xi)\le 
\nabla^k p(z)$
\item for every $\xi\in \Xi $, 
${a^k}(\xi,\cdot)\in C^{\infty}(\Omega)$.
\end{enumerate}
Furthermore, for $k=0$ and  $\epsilon\le 1/2$, then item (2)
can be replaced by 
$$
\sum_{\xi\in \Xi}  \bigl| {a^0}(\xi,z) \bigr|
\le 1+3\epsilon.$$
\end{proposition}
\begin{proof}
Fix $z\in \M$ and let $V_{z}= \mathcal{P}_m(\M)$ be endowed
with
the supremum norm from $C(z,\kappa h).$ We handle the cases $k\ge 1$ and $k=0$ separately.

By Lemma \ref{L_Markov}, for unit vectors $v_1,\dots,v_k\in T\M_{z}$,
the functional $$\lambda_{v_1,\dots,v_k}:p\mapsto \nabla^k p(z)(v_1,\dots ,v_k)$$
has norm $\|\lambda_{v_1,\dots,v_k}\|_{V_{z}'}\le C (\kappa h)^{-k} $,
so there is $\aaa^k:\Xi \times \Omega \to T_k^0\M$ for which the conditions for Lemma \ref{P_LPR} hold.
Consequently,  $\aaa^k$ satisfies the condition (\ref{E_general}) with $R= \frac{(1+\epsilon)}{\epsilon}\kappa h$,
$K= C (1-\epsilon)^{-1}(\kappa h)^{-k} $ and
$\opL =\nabla^k$.
Thus  (1)-(4) follow from Lemma \ref{L_regularity}.

In case $k=0$,  
the fact that $\delta_z$ has norm $\|\delta_z\|_{V_z'} =1$ follows automatically, and
Lemma \ref{P_LPR} ensures  $|\aaa^0(\xi,z)|\le (1-\epsilon)^{-1}\le 1+2\epsilon$ when $\epsilon\le 1/2$. 
By  Lemma \ref{L_regularity}, there exists $a^0:\Xi \times \Omega \to \R$ which satisfies
(1)-(4) with $\sum_{\xi\in\Xi}|a^0(z,\xi)|\le 1+3\epsilon$.
\end{proof}

 %
 %
 %

%
%
%
\subsection{RBF interpolation}
\label{SS_interp}
Given a positive definite
radial basis function (RBF) $\phi:\R^d\to \R$, 
we may restrict $\phi$ to $\M$, as in \cite{FW}, to 
obtain a positive definite kernel 
$$\Phi:\M\times \M\to \R:(x,y)\mapsto \phi(x-y).$$
Associated to $\Phi$ is a reproducing kernel Hilbert space $\Nn_{\Phi}\subset C(\M)$,
called the {\em native space}  
with the property that $\langle u,\Phi(\cdot, y)\rangle_{\Nn_{\Phi}} = u(y)$
for all $u\in \Nn_{\Phi}$ and $y\in\M$.

For finite $\Xi\subset \M$, we define the space $V(\Xi) = \mathrm{span}_{\xi\in \Xi} \Phi(\cdot,\xi)$.
The orthogonal projection  operator $I_{\Xi}:\Nn_{\Phi} \mapsto V_{\Xi}$
produces the unique function $I_{\Xi}f\in V_{\Xi}$ which satisfies $I_{\Xi}f|_{\Xi} = f|_{\Xi}$.
It follows that there exist Lagrange functions $\chi_{\xi}\in V_{\Xi}$ for which 
\revision{$\chi_{\xi}(\zeta) = \delta_{\zeta,\xi}$ where $\xi,\zeta\in \Xi$ and $\delta_{\zeta,\xi}$ is the
Kronecker $\delta$}.\revisioncom{Referee 2: Minor Comment 12}
Thus, we may write the interpolant as $I_{\Xi} f= \sum_{\xi\in \Xi} f(\xi) \chi_{\xi}$.

The {\em power function} 
$P_{\Xi}:
\M\to \R :  
x\mapsto 
\sup_{f\in \Nn_\Phi} 
\frac{|f(x) - I_{\Xi}f(x)|}{\|f\|_{\Nn_{\Phi}}}
$  
measures the 
pointwise interpolation error at $x$.
An elementary calculation shows that
$$
P_{\Xi}(x)^2:= 
\Phi(x,x) 
-2 \sum_{\xi\in\Xi} \chi_{\xi}(x) \Phi(x,\xi)
+\sum_{\xi,\zeta \in \Xi} \chi_{\xi}(x) \chi_{\zeta}(x)
\Phi(\xi,\zeta).
$$
This is the minimum value 
of the quadratic form $\mathcal{Q}_{\Xi,x}:\R^{\Xi}\times\R^{\Xi}\to \R$ 
defined as:
$$
\mathcal{Q}_{\Xi,x}(u) =  
\Phi(x,x) -2 \sum_{\xi\in\Xi} u_{\xi} \Phi(x,\xi)
+\sum_{\xi,\zeta \in \Xi}  u_\xi u_\zeta\Phi(\xi,\zeta).
$$
\revision{See Figure \ref{fig:torus_RBF} for an illustration of the power function on a curved torus.}\revisioncom{Referee 1: Minor correction 4.}

\begin{proposition}
\label{P_RBF}
If the kernel $\Phi:(x,y) \mapsto \phi(x-y)$ generated by the RBF $\phi$
satisfies $[\phi]_{k}:=\sup_{h>0}h^{-k}\revision{\inf}_{p\in \mathcal{P}_{k}} \|\phi-p\|_{L_{\infty}\bigl(B(0,h)\bigr)}<\infty$,\revisioncom{Changed $\sup$ to $\inf$}
if $\M$ is an algebraic manifold,
and if $\Omega\subset \M$  satisfies Assumption 2, 
then for $h=h(\Xi,\Omega)\le h_0$ we have 
$$\|P_{\Xi}\|_{L_{\infty}(\Omega)} \le C  h^{k/2} [\phi]_{k}.$$
\end{proposition}
For the condition  $[\phi]_{k}<\infty$, it is sufficient that $\phi$ is in the H{\"o}lder-Zygmund space $B_{\infty,\infty}^k(\R^N)$, and
a fortiori it suffices for  $\phi$ to be in $C^k(\R^d)$.
\begin{figure}[t]
\begin{tabular}{ccc}
\includegraphics[width=0.30\textwidth]{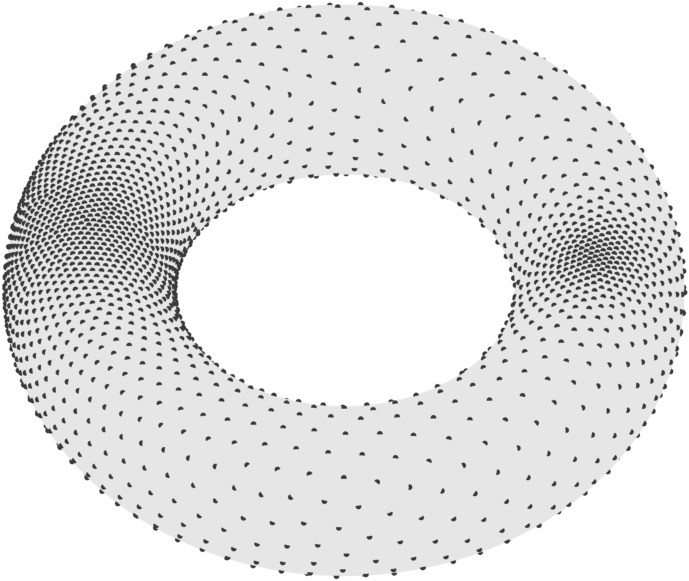} &
\includegraphics[width=0.30\textwidth]{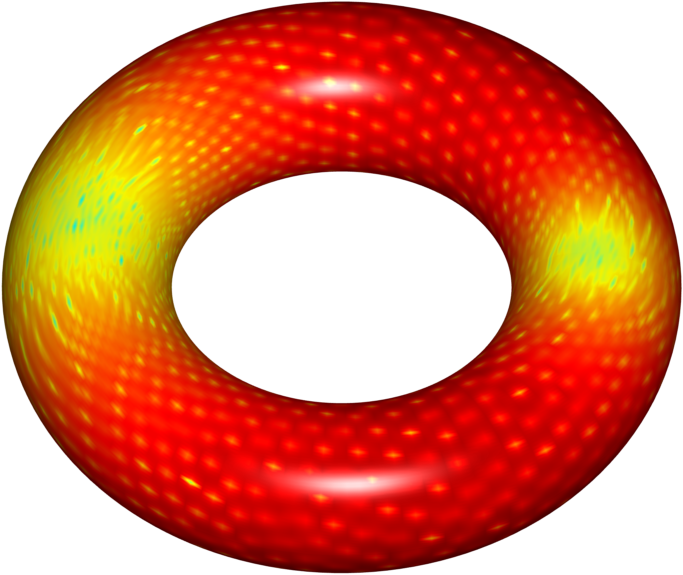} &
\includegraphics[width=0.33\textwidth]{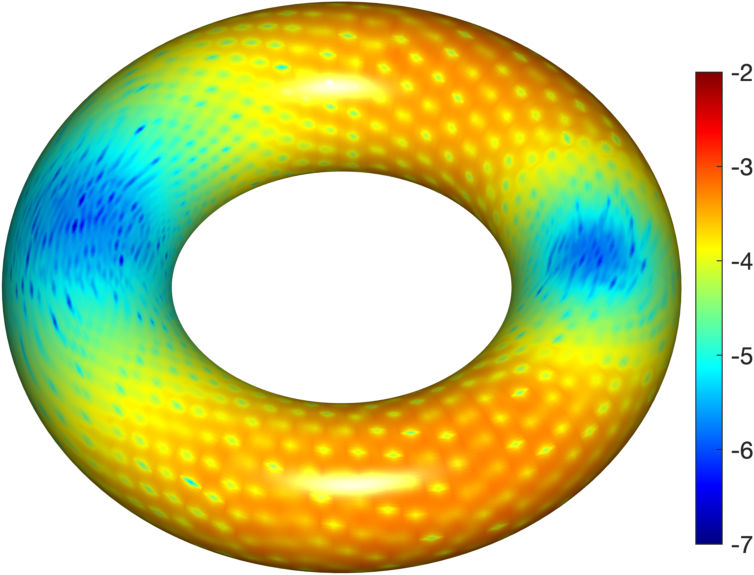}
\end{tabular}
\caption{Left: Distribution of centers $\Xi$ on the torus. Middle and right: heat map of the power function $P_{\Xi}$ (displayed on a $\log_{10}$ scale) using the Mat{\'e}rn kernel with $\matern=4$ and $\matern=5$, respectively.\label{fig:torus_RBF}}
\end{figure}
\begin{proof}
For the local polynomial reproduction 
$a$
from Proposition \ref{P_algebraic}
with $\epsilon =1/2$,
we note that $a(\cdot,x)\in\R^{\Xi}$, so
$P_{\Xi}(x)^2 
\le 
\mathcal{Q}_{\Xi,x}(a(\cdot,x))$
follows by optimality.
Since $\Phi$ is obtained as a restriction: 
$\Phi(x,y)= \phi(x-y)$, 
we have
$$\mathcal{Q}_{\Xi,x}(a(\cdot,x)) 
= 
\phi(0) 
- 2 \sum_{\xi\in\Xi} a(\xi,x) \phi(x-\xi)
+\sum_{\xi,\zeta \in \Xi}  
a(\xi,x) a(\zeta,x)\phi(\xi-\zeta).$$
Replacing every occurrence of $\phi$  
with $p\in \mathcal{P}_m$ in this expression 
we have
$$\tilde{\mathcal{Q}}_p := p(0) - 2 \sum_{\xi\in\Xi}a(\xi,x)  p(x-\xi)
+\sum_{\xi,\zeta \in \Xi}  a(\xi,x)    a(\zeta,x) p(\xi-\zeta)=0.$$
Thus $P_{\Xi}(x)^2\le  |\mathcal{Q}(a(\cdot,x)) | = |\mathcal{Q}(a(\cdot,x)) - \tilde{\mathcal{Q}_p}|$
holds, from which we obtain, by rearranging terms and applying the triangle inequality, the estimate
\begin{eqnarray*}
P_{\Xi}(x)^2
&\le& |\phi(0)-p(0)| +
4\max_{\xi\in\Xi\cap B_{\M}(x,2\kappa h)}|\phi(x-\xi) - p(x-\xi)| \\
&&+ \,
4\max_{\xi,\zeta\in \Xi\cap B_{\M}(x,2\kappa h)}|\phi(\xi-\zeta) - p(\xi-\zeta)|.
\end{eqnarray*}
The result now follows by noting that if 
$\dist_{\M}(\xi,\zeta)<4\kappa h$,
then \cite[Theorem 6]{FW} ensures
$|\xi-\zeta|\le K h$ for some $K\propto \kappa$, and the fact that
 $\|\phi - p\|_{B(0,  Kh)}\le C (Kh)^{k} [\phi]_{k}$.
 \end{proof}
 From this, we have the pointwise bounds on the interpolation error
  for $f\in \Nn_{\Phi}$
 \begin{equation}
 \label{Int}
 | f(x)- I_{\Xi} f(x)|\le P_{\Xi}(x) \|f\|_{\Nn_{\Phi}}
 \le C\Bigl(h(\Xi,\Omega)\Bigr)^{k/2}\|f\|_{\Nn_{\Phi}}.
 \end{equation}
 Although there exist kernel interpolation results 
 in more general settings 
 (e.g., \cite{FW} treats  compact, 
 embedded Riemannian manifolds), 
 such results generally 
 employ a global fill distance 
 $h_{\M} = \sup_{x\in \M} \dist_{\M}(x,\Xi)$, while
 the novelty of (\ref{Int}) is in its locality -- 
 the parameter 
 $h(\Xi,\Omega)$ 
 depends on the distribution of $\Xi\cap \Omega$ 
 in a potentially  small set $\Omega$.
 
This is illustrated by the numerical example
displayed in Figure \ref{fig:torus_RBF}, where
a finite point set $\Xi\subset \M$  is selected
on the torus $\M$ (which is the
zero set of a fourth degree polynomial in $\R^3$). 
Using the Mat{\'e}rn RBF 
of integer order $\matern$:
$$ 
\phi(x) 
= 
|x|^{\matern-3/2}
K_{\matern-3/2}(|x|)
$$
we display the power function $P_{\Xi}$ for two different 
values $\matern$,
noting it is significantly smaller in regions
$\Omega$ where $h(\Xi;\Omega)$ shrinks. 
The Mat{\'e}rn RBF serves as a fundamental solution (up to a constant multiple) of 
$(1-\Delta)^{\matern}$ in $\R^3$, and thus its native space is $W_2^\matern(\R^3)$. 
The restricted kernel 
$\Phi(x,y) =\phi(x-y)$
on the torus $\M$ has, by the trace theorem,
native space $W_2^{\matern-1/2}(\M)$.

We note that 
for the above choice of kernel
$\phi\in 
C^{\infty}(\R^3\setminus\{0\})$, 
$\phi\in C^{2\matern-3-\epsilon}(\R^3)$ 
for any $\epsilon>0$ but   
$\phi \notin C^{2\matern-3}(\R^3)$.
\revision{This and the fact that $[\phi]_{2s-3}<\infty$ 
(ensuring that
Proposition \ref{P_RBF} holds with $k=2\matern-3$)
follow from the 
asymptotic expansion
for the Mat{\'e}rn kernel given in \cite[Example 4.2]{HR}.
In short,  
$\phi(x) = u(|x|)+ \mathsf{h}_{2s-3}(x) v(|x|)$, where
 $u,v$ are analytic and even
 and
$$\mathsf{h}_{2s-3}(x)= \begin{cases} |x|^{2s-3},&2s-3\notin 2\Z\\
 |x|^{2s-3}|\log |x||, &2s-3\in 2\Z.\end{cases}$$
From this, the condition $[\phi]_{\alpha}<\infty$ is equivalent to  $[\mathsf{h}_{2s-3}]_{\alpha}<\infty$.
This  follows automatically for non-even indices; when $2s-3$ is even,
fix $h>0$ and set $p(x) = |x|^{2s-3}|\log h|$;
then simple algebraic manipulation shows $|\mathsf{h}_{2s-3}(x)-p(x)|\le h^{2s-3}|\mathsf{h}_{2s-3}(x/h)|$,
so $[\mathsf{h}_{2s-3}]_{2s-3}\le \|\mathsf{h}_{2s-3}\|_{L_{\infty}(B(0,1))}$.}\revisioncom{Referee 1: Fifth comment from ``Comments and suggestions''.}

\subsection{MLS approximation}\label{SS_MLS}
We consider
(\ref{eq:mlsprimal}),
modified in the following way. For an algebraic manifold $\M$,
with $\Omega\subset \M$, $\Xi\subset \Omega$ and $f\in C(\Omega)$,
we define the MLS approximant
\begin{align}
\label{Pol_MLS}
 \mathcal{M}_{\Xi}f(z) = 
    p_z^*(z),\text{ where } p_z^*:=\arg\min_{p \in \mathcal{P}_m} 
    \sum_{\xi \in \Xi} \left(p(\xi) -f(\xi) \right)^{2} \weightkernel(\xi,z)
\end{align}
by taking weight function
$\weightkernel(\xi,z) = \Phi(\frac{|z-\xi|}{\delta})$. Note that this is simply 
(\ref{eq:genmlsprimal}) with $\varPi=\mathcal{P}_m$.

\begin{remark}\label{R_euclidean_metric}
We note that 
this weight function satisfies Assumption \ref{A_weight}.
In particular, using the Euclidean norm on $\R^N$, we have
$F_{\xi}(z) := |z-\xi|$,
and can estimate $\nabla^k F_{\xi}$ by the metric equivalence (\ref{eq:expmap}).
Indeed, we have, 
$$
F_{\xi}(\Exp_{\xi}(x) )
= 
|\Exp_{\xi}(x)-\xi|
=
\left| \psi(x)\right|
$$
where $\psi(x) :=\Exp_{\xi}(x)-\xi $ is smooth 
(in fact, analytic since $\M$ is an algebraic manifold).
By the chain rule and product rule,
$D^{\alpha} F_{\xi}(\Exp_{\xi}(x) )$
is a linear  combination of functions of the form 
$h_{\beta}(\psi(x))\prod_{j=1}^{|\beta|} D^{\gamma_j}\psi(x)$
where $0<\beta\le \alpha$,
 $\sum_{j=1}^{|\beta|}\gamma_j = \alpha$,
and $h_{\beta}$ is the partial derivative of the Euclidean norm,
$D^{\beta} (|y | )= h_{\beta}(y)$.
\revision{Each $h_{\beta}$ is homogeneous of order $1-|\beta|$ and satisfies
 $|h_{\beta}(y)|\le C_{|\beta|,N}|y|^{1-|\beta|}$, with a constant $C_{\beta,N}$ depending on $|\beta|$ and $N$.}\revisioncom{Added a bit of explanation.}
Likewise, 
 $\psi$ is smooth and satisfies $\bigl|D^{\gamma}\psi(x)\bigr|\le C_{\gamma,\Omega}$,
 it follows that 
 $$
 \left| h_{\beta}(\psi(x))\prod_{j=1}^{|\beta|} D^{\gamma_j}\psi(x)\right|
 \le 
\revision{ C_{|\beta|,\Omega} |\psi(x)|^{1-|\beta|},}\revisioncom{Fixed an error with the  exponent}
 $$
 and, hence, 
$|D^{\alpha} F_{\xi}(\Exp_{\xi}(x) )|\le C |\psi(x) |^{1-|\alpha|} = C|\Exp_{\xi}(x)-\xi|^{1-|\alpha|}$.
Finally,
the metric equivalence (\ref{eq:Riemannian-Euclidean})
guarantees that (\ref{eq:homogeneity}) is satisfied.
 \end{remark}

\begin{proposition}
\label{P_MLS_algebraic}
For an algebraic manifold $\M$ and
 a compact subset $\Omega\subset \M$ with Lipschitz boundary,
 a $C^k$ weight  kernel
 $
 \weightkernel(\xi,z) 
 = 
 \Phi\Bigl(\frac{F_{\xi}(z)}{\delta}\Bigr),
 $
  there is a constant $C$  so that for
 $z\in \Omega$ 
and $j\le k$ the MLS shape functions satisfy
 $$
 \sum_{\xi\in\Xi} \|\nabla^j b^{\star}(\xi,z)\|_{T_{j}^0\M_z} \le C 
 \delta^{-j}  \revision{\mathrm{card}}\bigl(\Xi \cap B(z,\delta)\bigr).
 $$
\end{proposition}

From this, we have the following results, which match 
well-known results from the Euclidean setting, but are novel in the context
of MLS approximation on manifolds. 
In the  statement of the result, as well as its proof,
we use the notation $D_v f(z)$ to denote the directional
derivative of $f$ in the direction $v$ at $z$ ; for 
functions defined on a $\R^N$ neighborhood of $z$, this 
is the customary $\lim_{t\to 0} \frac1t (f(z+tv)-f(z))$,
while for functions defined on $\M$ and $v\in T_0^1\M_z$, this is $\nabla f(z) v$.
Higher order approximation results follow with a straightforward modification.
\begin{proposition}
\label{prop:mls_error_estimate}
For 
\revision{an algebraic manifold $\M$,}\revisioncom{Referee 2: Minor Comment 13}
a compact subset $\Omega\subset \M$ with Lipschitz boundary and for any $\rho_0>1$, there is a constant $C$ so that,
for any sufficiently dense  $\Xi\subset \Omega$ which satisfies $\rho := \frac{h}{q}\le \rho_0$, we
have for any $k\le m+1$ and $f\in C^k(\Omega)$, that
$$
\left| f(z) -\mathcal{M}_{\Xi}f(z)\right| 
\le 
C  \delta^k \|f\|_{C^k(\Omega)}
$$
and for any unit tangent vector $v\in T_0^1\M_{z}$  at $z$, we have 
$$
\left| 
D_v f(z)
-
D_v \mathcal{M}_{\Xi}f(z)
\right| \le C\delta^{k-1} \|f\|_{C^k(\Omega)}$$
\end{proposition}
\begin{proof}
We  have the standard estimate
%
\begin{align*}
  \left| f(z) -\mathcal{M}_{\Xi}f(z)\right| 
  &\le 
    \left| f(z) -p(z)\right|+\left|\mathcal{M}_{\Xi} p(z) - \mathcal{M}_{\Xi}f(z)\right|\\
  &\le    \left(1+\sum_{\xi \in \Xi}| b^{\star}(\xi,z) | \right) \inf_{p\in  \mathcal{P}_m }
    \|f-p\|_{L_{\infty}(\Omega\cap B(z,\delta))}.
\end{align*}
%
Note that $\sum_{\xi\in\Xi}| b^{\star}(\xi,z) | $ is bounded by Lemma \ref{L_MLS_stability}.
If $f\in C^k(\Omega)$, then we can take an extension $\mathcal{E}f\in C^k(W)$ defined on a neighborhood $W$ of $\Omega$ in $\R^N$ 
 with $\|\mathcal{E}f\|_{C^k(W)}\le C\|f\|_{C^k(\Omega)}$.
Thus,
%
$$
\|f-p\|_{L_{\infty}(\Omega\cap B(z,\delta))}
\le 
\|\mathcal{E}f-p\|_{L_{\infty}(W\cap B_{\euc}(z,\delta))} 
\le 
\delta^k\|\mathcal{E}f\|_{C^k(W)} 
\le 
\delta^k \|f\|_{C^k(\Omega)}.
$$
In a similar way, we may write
$$
|\nabla f(z)v - \nabla \mathcal{M} f(z)v|
=
|\nabla f(z)v -\nabla p(z)v|
+\| \nabla \mathcal{M}( f-p)(z)\|_{T_1^0\M_z}.
$$
%
The second expression is bounded by 
$\sum_{\xi\in\Xi}\| \nabla b^{\star}(\xi,z) \| \|f-p\|_{L_{\infty}(\Omega\cap B(z,\delta))}$,
while the first can be controlled by directional derivative
$|D_v \mathcal{E}f(z)-D_v p(z)|$
since both functions $\mathcal{E}f$ and $p$ are defined on $W$.
Thus,
%
\begin{eqnarray*}
|\nabla f(z)v - \nabla \mathcal{M} f(z)v|
&\le&
\inf_{p\in \mathcal{P}_{m}} |D_v \mathcal{E}f(z)-D_v p(z)|
+
\frac{C}{h}  \| \mathcal{E}f-p\|_{L_{\infty}(W\cap B_{\euc}(z,\delta))}\\
&\le& 
C\delta^{k-1}\|f\|_{C^k(\Omega)}
\end{eqnarray*}
%
and the result follows.
\end{proof}

%
%
%
\subsubsection{Treating noisy data}
\label{SS_noisy}
We note that (\ref{Pol_MLS}) can treat data $Y= (\xi,y_{\xi})_{\xi\in\Xi}$  directly. 
If $y_{\xi} = f(\xi) + \epsilon_{\xi}$ is
perturbed by noise, say  by  independent, identically distributed $\epsilon_{\xi} \sim \mathcal{N}(0,\sigma^2)$,
then we may write
%
\begin{align*}
	\mathcal{M}_{\Xi}(Y)(z)&:=\sum_{\genfrac{}{}{0pt}{}{\xi \in \Xi}{\weightkernel(\xi,z)>0}} b^{\star}(\xi,z) y_{\xi} = \sum_{\genfrac{}{}{0pt}{}{\xi \in \Xi}{\weightkernel(\xi,z)>0}} b^{\star}(\xi,z) f(\xi) +  \sum_{\genfrac{}{}{0pt}{}{\xi \in \Xi}{\weightkernel(\xi,z)>0}} b^{\star}(\xi,z) \epsilon_{\xi}
	\\&=\mathcal{M}_{\Xi}(f)(z)+\sum_{\genfrac{}{}{0pt}{}{\xi \in \Xi}{\weightkernel(\xi,z)>0}} b^{\star}(\xi,z) \epsilon_{\xi}
\end{align*}
%
Now using linearity of expectation,  and the fact that $\mathbb{E}\left[ \epsilon_{\xi}  \right]=0$, we get
%
\begin{align*}
	\mathbb{E}\left[ \left(f(z)- \mathcal{M}_{\Xi}(Y)(z)\right)^2  \right]
	&=\left( f(z)- \mathcal{M}_{\Xi}(f)(z)\right)^2 
	+\mathbb{E}\left[ \left( \sum_{\genfrac{}{}{0pt}{}{\xi \in \Xi}{\weightkernel(\xi,z)>0}} b^{\star}(\xi,z) \epsilon_{\xi} \right)^2  \right].
\end{align*}
%
Expressing 
$\bigl( 
\sum_{\genfrac{}{}{0pt}{}{\xi \in \Xi}{\weightkernel(\xi,z)>0}} b^{\star}(\xi,z) \epsilon_{\xi} 
\bigr)^2$
as
$\sum_{\genfrac{}{}{0pt}{}{\xi,\xi^{\prime} \in \Xi}{\weightkernel(\xi,z) \weightkernel(\xi^{\prime},z)>0 }} 
		 b^{\star}(\xi,z) b^{\star}(\xi^{\prime},z)\epsilon_{\xi}  \epsilon_{\xi^{\prime}} 
$,
we treat the second term as 
%
\begin{align*}
 \mathbb{E}\left[ \left( \sum_{\genfrac{}{}{0pt}{}{\xi \in \Xi}{\weightkernel(\xi,z)>0}} b^{\star}(\xi,z) \epsilon_{\xi} \right)^2  \right]
 &= 
 \sigma^{2} 
 \sum_{\genfrac{}{}{0pt}{}{\xi \in \Xi}{\weightkernel(\xi,z)>0}} |b^{\star}(\xi,z)  |^2 
 \le
 \sigma^{2}\left( \sum_{\xi \in \Xi} |b^{\star}(\xi,z)| \right)^2
 \le 16 C^2 \rho^{2d} \sigma^2.
 \end{align*}
%
This shows for the mean squared error (MSE) the bound
\begin{align*}
	\mathbb{E}\left[ \left(f(z)- \mathcal{M}_{\Xi}(Y)(z)\right)^2  \right] \le \left( f(z)- \mathcal{M}_{\Xi}(f)(z)\right)^2+16 C^2 \rho^{2d} \sigma^2.
\end{align*}

%
%
%
\section{Numerical experiments for MLS}
\label{S_NE}
In this section, we numerically demonstrate the MLS approximation stability and error estimates from the previous section. Before presenting these results, we discuss the algorithm used to produce the MLS approximants.

\subsection{MLS method for approximation on manifolds\label{sec:mls_algorithm}}
We describe our method for
constructing MLS approximants for two dimensional algebraic manifolds  $\M\subset \R^3$,
the generalization to higher dimensional manifolds should be straightforward. To make
the method general and easy to implement, we want to work with the ambient polynomial
space $\mathcal{P}_m(\R^3)$ of total degree $m$. The issue with this is that depending on $m$ and
the degree of the algebraic manifold, this can lead to rank deficiency in the local least
squares problems inherent to \eqref{Pol_MLS} since the dimension of $\mathcal{P}_m(\R^3)$  
may be larger than
the dimension of $\mathcal{P}_m(\M)$. 
To further elucidate the issue we first work out the the exact dimension $\mathcal{P}_m(\M)$ using some
results from algebraic geometry in Appendix \ref{S_ideal}.

\subsubsection*{Selecting the parameter $\delta$}

The support parameter defines the local subsets of points of $\Xi\subset\Omega$ for the weighted least squares problem for each evaluation point $x\in X\subset\Omega$; we denote these subsets as $\Xi_{x}(\delta)$.  Given $m$ and the set of evaluation points $X$ to compute the approximant, $\delta$ should be chosen so that 
    $N_{\Xi_{x}} = \revision{\mathrm{card}} \Xi_{x}(\delta)$
is larger than the dimension of the polynomial space being used in the approximation.  

Since we use the ambient space $\mathcal{P}_m(\R^3)$ to construct the MLS approximants, we require that $N_{\Xi_{x}} > \dim(\mathcal{P}_m(\R^3))=M$.  To ensure that this holds, we use the following procedure.  For each $x$ we determine the $2M$ nearest neighbors from $\Xi$ to $x$ and follow this by computing the minimum radius of the ball in $\R^3$ centered at $x$ that contains all these nearest neighbors. We then choose $\delta$ as the maximum of these minimum radii for all $x\in X$. In our experiments this guaranteed that $N_{\Xi_{x}} \geq 2M$, for all $x$.  Note that determining $\delta$ with this procedure can be done efficiently using $k$-d tree.

\subsubsection*{Singular value decomposition}
To deal with the issue of rank deficiency of MLS problem \eqref{Pol_MLS} that can arise from a dimension mismatch between $\mathcal{P}_m(\R^3)$ and $\mathcal{P}_m(\M)$, we use the singular value decomposition (SVD) of the Vandermonde matrix formed from evaluating a basis for $\mathcal{P}_m(\R^3)$ at $\Xi_x(\delta)$. An important step in this process is choosing the basis for $\mathcal{P}_m(\R^3)$, for which we propose using one that depends on $\Xi_x(\delta)$ as follows.  Let $\{p_j\}_{j=1}^M$ denote the standard monomial basis for $\mathcal{P}_m(\R^3)$ and select the basis for the MLS approximant on $\Xi_x(\delta)$ as $\{p_j((\cdot-x)/\delta)\}_{j=1}^{M}$.  We denote the Vandermonde matrix formed by evaluating this basis at $\Xi_{x}$ as $\mathbf{P}_{\Xi_{x}}$.  

Using the the procedure for choosing $\delta$ described above, this Vandermonde matrix is overdetermined, with $M$ columns and $N_{\Xi_{x}} \geq 2M$ rows.  We denote the (reduced) SVD of $\mathbf{P}_{\Xi_{x}}$ as
\begin{align*}
    \mathbf{P}_{\Xi_{x}} = \mathbf{U}_{\Xi_{x}} \boldsymbol{\Sigma}_{\Xi_{x}} \mathbf{V}_{\Xi_{x}}^T,
\end{align*}
where the $M$ columns of $\mathbf{U}_{\Xi_{x}}$ are orthonormal, $\mathbf{V}_{\Xi_{x}}$ is an $M$-by-$M$ orthogonal matrix, and $\boldsymbol{\Sigma}_{\Xi_{x}}=\text{diag}(\sigma_1,\ldots,\sigma_M)$. To determine a discretely orthonormal basis for $\mathcal{P}_m(\Xi_{x}(\delta))$, we use the first $K$ columns of $\mathbf{U}_{\Xi_{x}}$ corresponding to the numerical rank of $\mathbf{P}_{\Xi_{x}}$.\footnote{This value is determined from the number of singular values that are greater than $N_{\Xi_{x}}\sigma_1 \epsilon_M$, where $\sigma_1$ is the largest singular value and $\epsilon_M = 2^{-52}$ is the machine $\epsilon$ for double precision floating point arithmetic.  This is a similar metric as used by the \texttt{rank} function in MATLAB.}  This discrete orthonormal basis, together with the similarly reduced $\mathbf{V}_{\Xi_{x}}$ and $\boldsymbol{\Sigma}_{\Xi_{x}}$ can then be used directly to solve the optimization problem \eqref{Pol_MLS}.

\revision{We note that the dimension $d$ of ambient space enters the numerical complexity in performing the SVD. To this end, note that $\mathbf{P}_{\Xi_{x}} \in \R^{N_{\Xi_x} \times M}$ with $M=\dim \mathcal{P}_m(\R^N) = \genfrac{(}{)}{0pt}{}{m+N}{m}$. To give a numerical complexity, we note that we are in the situation $N_{\Xi_x}\ge 2M$, thus we refer to the algorithm presented in \cite[Section 5]{Chan:1982} and obtain a complexity of $\mathcal{O}\left(N_{\Xi_{x}} M^2 + M^3 \right)$. We are mostly concerned with applications where $N=3$. In other applications (e.g., machine learning) it often the case that $N\gg1$. This situation most likely needs different approaches, as the ambient space might in those cases be too high dimensional to consider the full polynomial space directly.}\revisioncom{Ref. 1, Comment 6}

Since $N_{\Xi_{x}}$ is much less that $N_{\Xi}$ and does not grow as $\Xi$ increases, the procedure is efficient (and is pleasingly parallel). It also allows one to work entirely with ambient polynomial space and the algorithm finds out both the (numerical) dimension of the space of restricted polynomials and a basis for this space automatically.

\subsubsection*{Weight function}
The final piece for the MLS problem is the weight function, for which we use the $C^4(\R^3)$ Wendland kernel
\begin{align*}
\mlswght(\xi,\zeta) = \lp 1-\delr \rp_{+}^6 \lp 1 + 6\lp\delr\rp + \frac{35}{3}\lp\delr\rp^2\rp.
\end{align*}

\begin{figure}
\centering
\begin{tabular}{cc}
\includegraphics[width=0.45\textwidth]{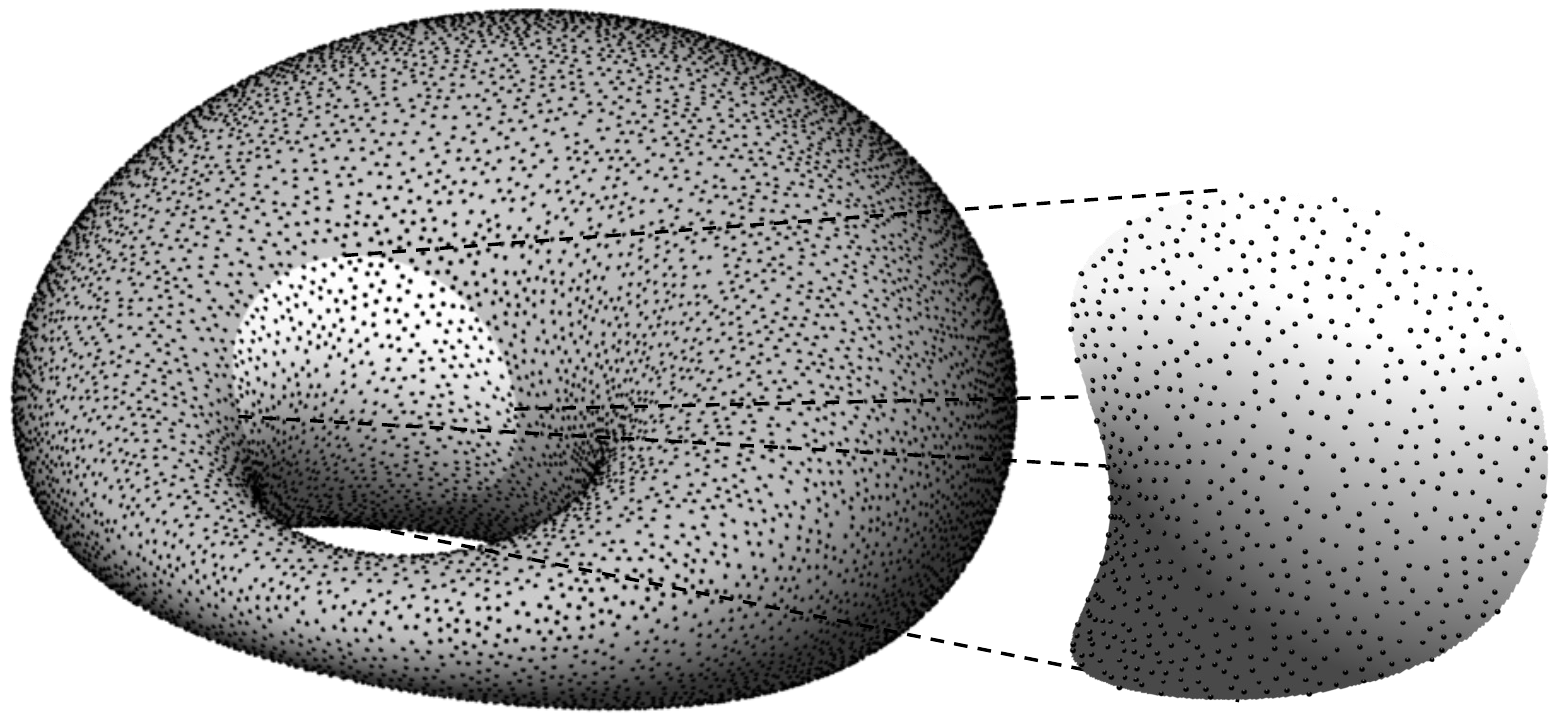} & \includegraphics[width=0.3\textwidth]{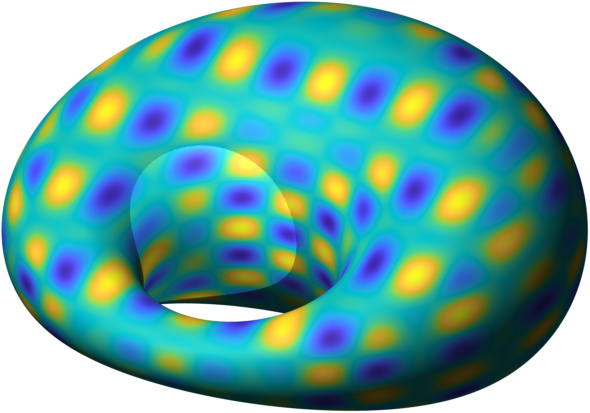}\\ 
\end{tabular}
\caption{Left panel: visualization of the manifold $\M$ and the compact subset $\subdom\subset\M$ (inset) used in the MLS numerical experiments; see \eqref{eq:cyclide} and \eqref{eq:patch} for exact definitions.  Black solid spheres mark the node sets $\Xi$ on $\M$ and $\subXi$ on $\subdom$. Right panel: heat map of the target function on $\M$ considered in the numerical experiments; the brighter region highlights the target function on $\subdom$.\label{fig:cyclide}}
\end{figure}

\subsection{Verification of error and stability estimates}\label{SS:cyclide}
\begin{figure}
\centering
\begin{tabular}{cc}
\includegraphics[width=0.47\textwidth]{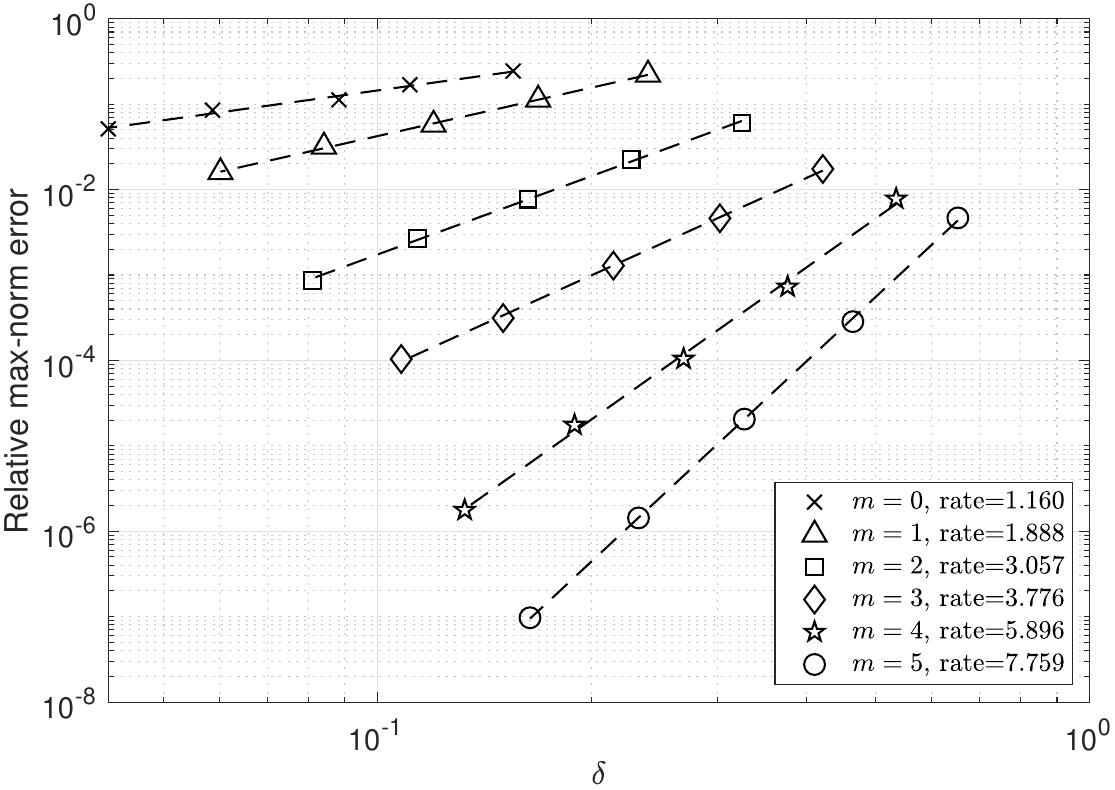} & 
\includegraphics[width=0.47\textwidth]{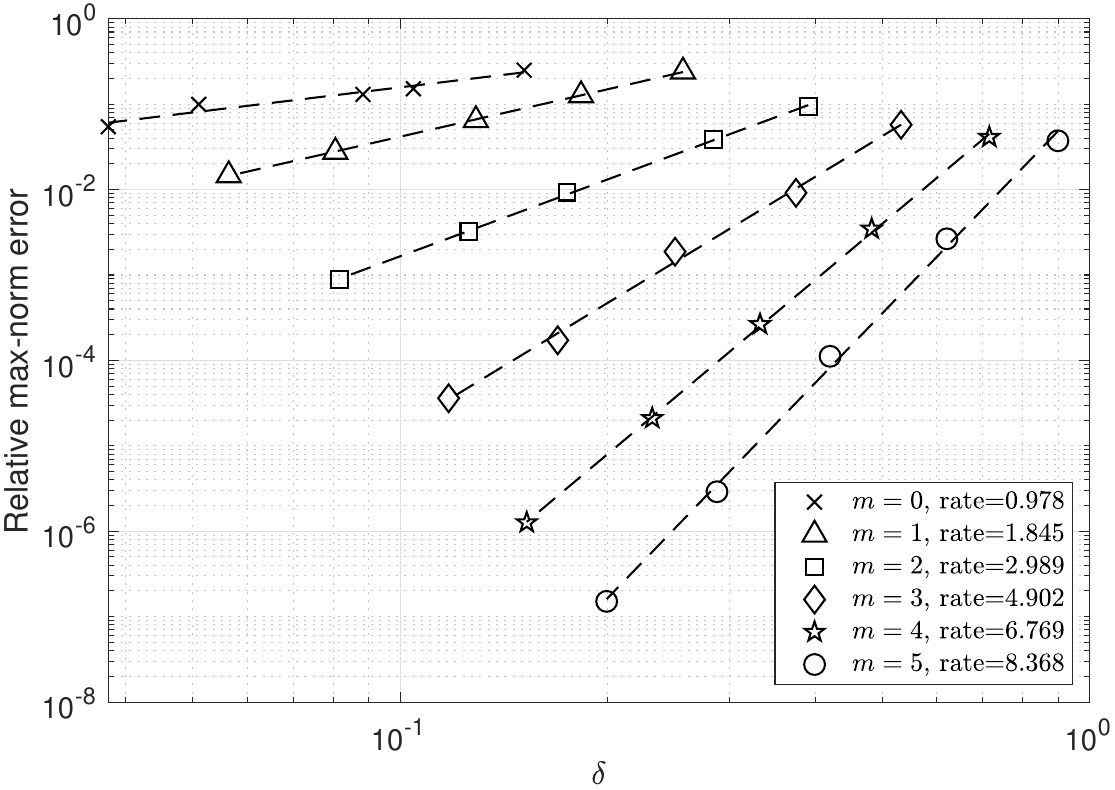} \\
(a)  & (b) 
\end{tabular}
\caption{Convergence results where the approximation problem is done over (a) all of $\M$ and (b) only the compact subset $\subdom$.  The estimated rates of convergence for each $m$ are included in legend labels.\label{fig:convergence}}
\end{figure}
Consider the following two-dimensional, degree four algebraic manifold:
\begin{align}
\M = \left\{x = (x_1,x_2,x_3) \in \R^3 \;\left| \lp x_1^2 + x_2^2 + x_3^2 - d^2 + b^2 \rp^2 - 4\lp a x_1 + c d \rp^2 - 4 b^2 x_2^2 = 0 \right. \right\},
\label{eq:cyclide}
\end{align}
where $a=2$, $b=1.9$, $d=1$, and $c^2 = a^2 - b^2$; this manifold is an example of a \textit{cyclide of Dupin}. We test the stability and error estimates of MLS approximations over the whole manifold, i.e., $\Omega=\M$, and over the compact subset of $\M$ defined as
\begin{align}
\subdom = \M \cap B_1(\xi_{\rm c}),
\label{eq:patch}
\end{align}
where $B_1(\xi_{\rm c})$ is the ball in $\R^3$ of radius 1 centered at the point $\xi_{\rm c} = (0,\sqrt{9-6a+b^2},0)$ on the cyclide; see the left panel of Figure \ref{fig:cyclide} for an illustration of $\M$ and $\subdom$.  For the target function, we use the following $C^{\infty}$ function defined in $\R^3$ and restricted $\M$
\begin{align}
f(x) = \cos(\pi(x_1-\frac{3}{10}))\sin(2\pi(x_2-\frac15))\cos(3\pi(x_3-\frac{1}{10}));
\label{eq:target_function}
\end{align}
see the right panel of Figure \ref{fig:cyclide} for an illustration of $f$.  

To test the error estimates from Proposition \ref{prop:mls_error_estimate}, we consider MLS approximants constructed from samples of \eqref{eq:target_function} at unstructured, quasi-uniform point sets $\Xi\subset\M$ of cardinalities $N_{\Xi} = 2^{i}$, $i=14,\ldots,18$.  We also consider approximants of \eqref{eq:target_function} from restrictions of these point sets to $\subdom$, i.e., $\subXi:=\Xi \cap \subdom$, which results in point sets of cardinalities $N_{\subXi}=706,\,1399,\,2838,\,5655,\,11308$. See the left panel of Figure \ref{fig:cyclide} for an illustration of these points with $N_{\Xi}=2^{14}$ and $N_{\subXi}=706$ nodes.  The error in MLS approximants for $\Omega$ are computed at a finer set of $N_{X}=2^{21}$ evaluation points $X \subset \M$, while approximants on $\subdom$ are computed at $N_{\subX}=86123$ evaluation points  $\subX=X\cap \subdom$.

The convergence results using the set-up described above are given in Figure \ref{fig:convergence} for polynomial degrees $m=0,1,\ldots,5$. Included in these results are the estimated rates of convergence computed from a line of best fit of the (log) of the data.  The convergence rates are estimated in terms of $\delta$, which should be proportional to the mesh-norm of $\Xi$ since these point sets are quasi-uniform.  We see for the MLS approximants over all of $\M$ that the estimated convergence rates are close to the optimal rates of $m+1$ expected from Proposition \ref{prop:mls_error_estimate} in the case of $m=0,1,2,3$, but that for $m=4$ and $5$ the rates are higher by about 1.  Similar results hold for MLS approximants over $\subdom$, but in this case we observe higher rates also for $m=3$.
These numerical results back up the new theoretical results that one can use MLS entirely locally and still obtain similar convergence rates to working globally on a manifold.

We note that the cyclide surface is defined by the level surface of an irreducible polynomial of degree $k=4$.  Thus, by (\ref{Hypersurface_dimension}) above, the dimension of $\mathcal{P}_m(\M)$ is given by: $$\dim(\mathcal{P}_m(\M)) =  \begin{cases} 2m^2+2&m\ge 4\\
\frac16 (m+3)(m+2)(m+1)&m < 4.\end{cases} .$$
In the numerical experiments from Figure \ref{fig:convergence} (as well as many others not presented here), we observed that the SVD procedure described in Section \ref{sec:mls_algorithm} consistently produced a discrete orthogonal basis for each $\mathcal{P}_m(\Xi_x(\delta)))$ that  matched the expected dimension from  the formula above. This numerical evidence supports the robustness and generality of this relatively simple and straightforward method of working with polynomials in the ambient space.

\begin{figure}
\centering
\includegraphics[width=0.5\textwidth]
{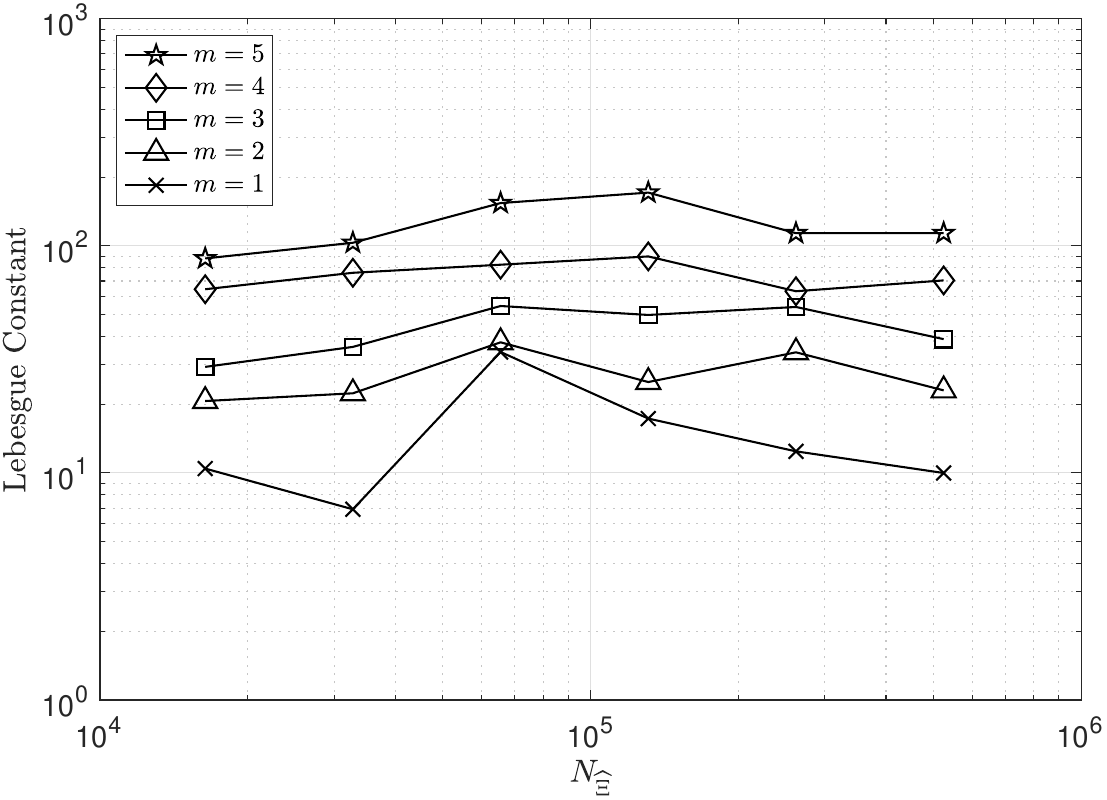} 
\caption{Numerically computed Lebesgue constants for the MLS approximation problem for the compact subset $\subdom$ of cyclide using quasi-uniform points of increasing cardinality $N_{\subXi}$.\label{fig:lebesgueconst}}
\end{figure}

\begin{figure}
\centering
\includegraphics[width=0.95\textwidth]{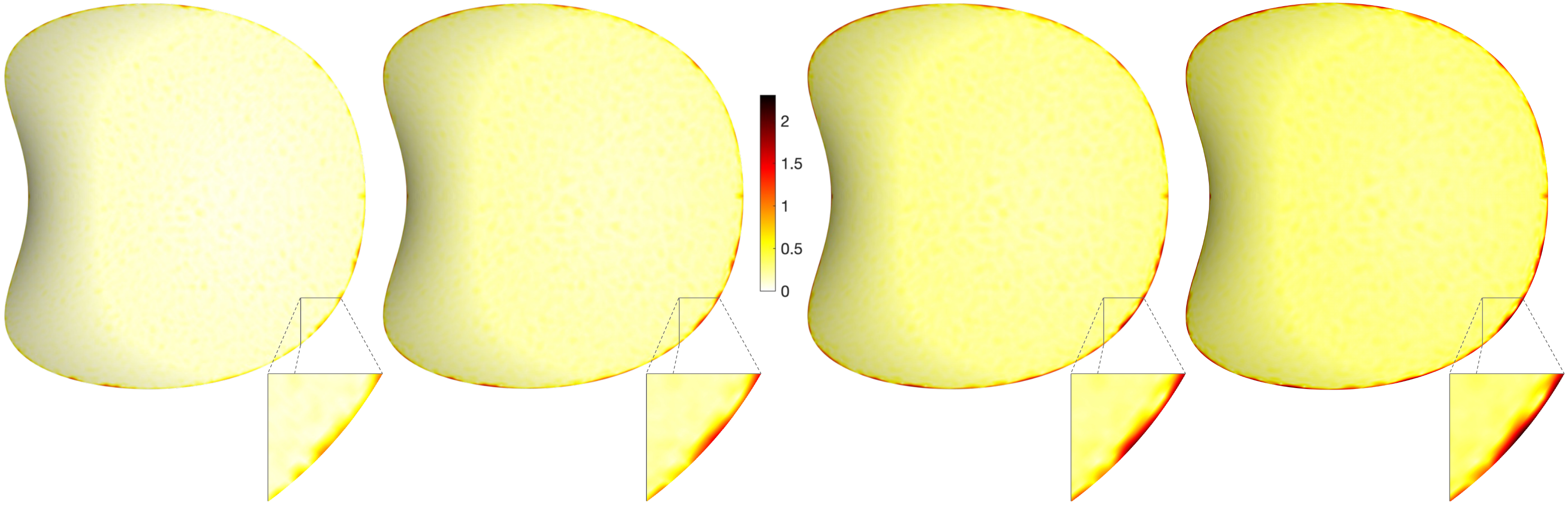}
\caption{Lebesgue functions of the MLS approximants on $\subdom$ using a point set with $N_{\subXi}=5655$ and (from left to right) polynomial degrees $m=2,3,4,5$.  Note that the color scheme is same for each plot according to the given colorbar and is based on the $\log_{10}$ of the values of the Lebesgue functions. \label{fig:lebesguefunc}}
\end{figure}

\begin{table}[tb]
\centering
\begin{tabular}{|c||c|c|c|c|}
\hline 
& \multicolumn{2}{c|}{Perturbation $\sigma=10^{-2}$} & \multicolumn{2}{c|}{Perturbation $\sigma=10^{-1}$} \\
\hline
Deg.\ & Mean diff.\ & Std Dev & Mean diff.\ & Std Dev\\
\hline
\hline
\multicolumn{5}{|c|}{Approximation on $\Omega$} \\
\hline
$0$ & $4.49\times 10^{-2}$ & $2.94\times 10^{-3}$ & $4.53\times 10^{-1}$ & $2.85\times 10^{-2}$ \\
$1$ & $8.59\times 10^{-2}$ & $2.68\times 10^{-2}$ & $8.77\times 10^{-1}$ & $2.70\times 10^{-1}$ \\
$2$ & $5.05\times 10^{-2}$ & $4.25\times 10^{-3}$ & $5.07\times 10^{-1}$ & $5.84\times 10^{-2}$ \\
$3$ & $4.50\times 10^{-2}$ & $3.26\times 10^{-3}$ & $4.46\times 10^{-1}$ & $3.86\times 10^{-2}$ \\
$4$ & $4.04\times 10^{-2}$ & $2.91\times 10^{-3}$ & $4.04\times 10^{-1}$ & $2.78\times 10^{-2}$ \\
$5$ & $3.67\times 10^{-2}$ & $2.36\times 10^{-3}$ & $3.61\times 10^{-1}$ & $2.16\times 10^{-2}$ \\  
\hline
\multicolumn{5}{|c|}{Approximation on $\subdom$} \\
\hline
$0$ & $3.89\times 10^{-2}$ & $2.78\times 10^{-3}$ & $3.88\times 10^{-1}$ & $3.23\times 10^{-2}$ \\
$1$ & $3.14\times 10^{-2}$ & $3.18\times 10^{-3}$ & $3.07\times 10^{-1}$ & $3.11\times 10^{-2}$ \\
$2$ & $3.01\times 10^{-2}$ & $2.21\times 10^{-3}$ & $2.95\times 10^{-1}$ & $2.56\times 10^{-2}$ \\
$3$ & $2.69\times 10^{-2}$ & $2.31\times 10^{-3}$ & $2.67\times 10^{-1}$ & $2.12\times 10^{-2}$ \\
$4$ & $2.60\times 10^{-2}$ & $2.57\times 10^{-3}$ & $2.61\times 10^{-1}$ & $2.72\times 10^{-2}$ \\
$5$ & $2.53\times 10^{-2}$ & $2.58\times 10^{-3}$ & $2.59\times 10^{-1}$ & $3.13\times 10^{-2}$ \\
\hline
\end{tabular}  
\caption{Stability of the MLS approximation under noise introduced to the target function \eqref{eq:target_function}.  Noise was added to the samples of the target function as $\tilde{f}_j = f_j + \epsilon_j$, where $\epsilon_j \sim \mathcal{N}(0,\sigma)$.  MLS  approximants of $\tilde{f}_j$ were computed using 100 different trials and results show the mean and standard deviation of max-norm differences between the MLS approximants and the exact function.  Computations for $\Omega$ were done using $\subXi$ with $N_{\Xi}=262144$ and the errors were measured at $2097152$ points over $\Omega$, while computations for $\subdom$ were done using $\subXi\subset\subdom$ with $N_{\subXi}=11308$ and the errors were measured at $86123$ points over $\subdom$\label{tbl:stability}}
\end{table}

We conclude with some experiments on the stability of the MLS approximants.  These are illustrated as follows:
\begin{enumerate}
    \item Table \ref{tbl:stability} displays results of experiments where the data is perturbed by different levels of noise. As expected from the discussion in Section \ref{SS_noisy}, we see that the errors in the MLS approximants are closely controlled by the noise level.
    \item Figure \ref{fig:lebesgueconst} displays the numerically computed Lebesgue constants for the MLS approximants on $\subdom$.  We see that the Lebesgue constants appear to remain bounded as expected from the results given in Section \ref{sec:shape_functions}.
    \item Figure \ref{fig:lebesguefunc} gives visualizations of the the Lebesgue functions on $\subdom$ for different polynomial degrees. These images show that the Lebesgue functions are largest along the boundary of the domain and increase with increasing $m$, which is also the typical behavior in planar domains.
\end{enumerate}

\subsection{Meshed surface}
\begin{figure}
\centering
\begin{tabular}{cc}
\includegraphics[width=0.3\textwidth]{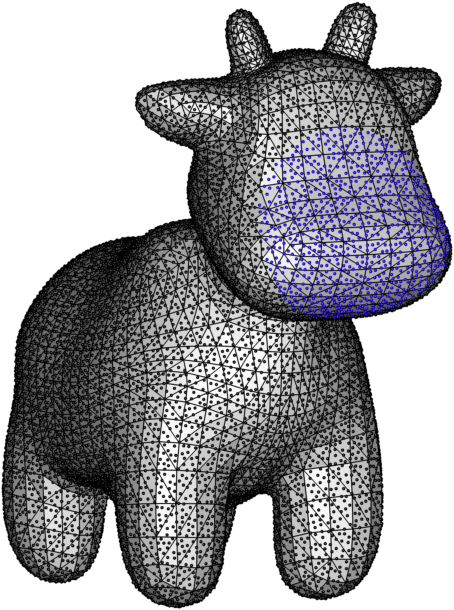} & \includegraphics[width=0.33\textwidth]{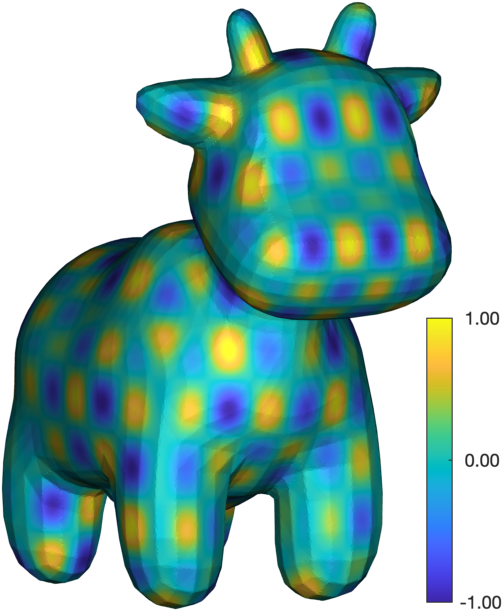}\\
(a) & (b) \\
\includegraphics[width=0.33\textwidth]{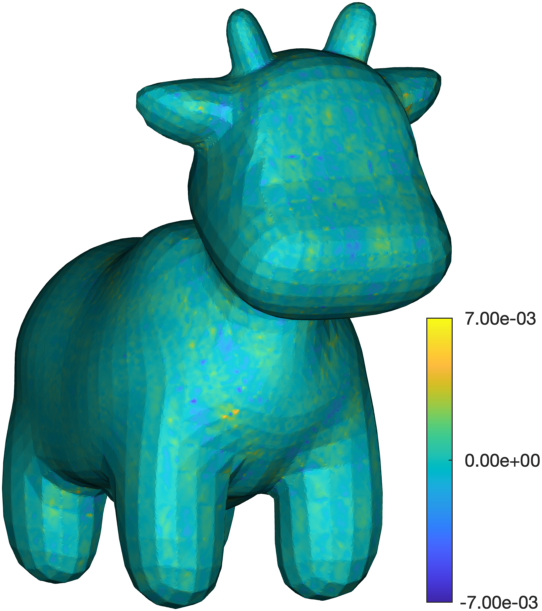} & \includegraphics[width=0.33\textwidth]{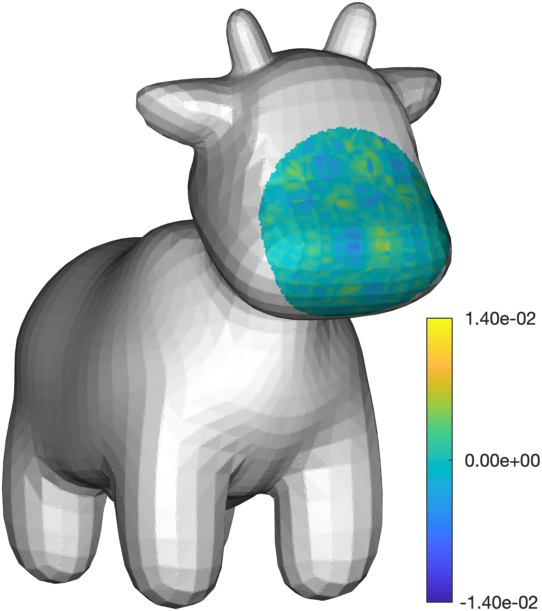} \\
(c) & (d)
\end{tabular}
\caption{(a) Visualization of the piecewise algebraic Spot manifold $\M$ and the compact subset $\subdom\subset\M$ (highlighted by the blue points) in the MLS numerical experiments. (b) Heat map of the target function on $\M$ considered in the numerical experiments. (c) \& (d) Heat map of the errors in the $m=3$ MLS approximants of the target function when using all points on $\M$ a restrictions of the points to $\subdom$.\label{fig:spot}}
\end{figure}
We lastly explore the behavior of the MLS method on a manifold that is not covered by our theory. For this we use the piecewise algebraic ``Spot'' manifold $\M$ from~\cite{crane2013robust}, which is defined by a triangulated mesh consisting of 2930 vertices. We employ the same algorithmic approach as the previous section. That is, we use global polynomials restricted to the surface, and again  
consider both the whole manifold, i.e., $\Omega=\M$ and a compact subset of $\M$ defined as
\begin{align}
\subdom = \M \cap B_{\frac{3}{10}}(\xi_{\rm c}),
\label{eq:patchspot}
\end{align}
where $B_{\frac{3}{10}}(\xi_{\rm c})$ is the ball in $\R^3$ of radius $\frac{3}{10}$ centered on the nose of Spot.  See Figure \ref{fig:spot} (a) for a visualization of Spot and the subdomain $\subdom$.

We consider MLS approximations of the target function
\begin{align}
f(x) = \cos(6\pi x_1)\sin(9\pi(x_2-\frac{1}{10}))\cos(7\pi x_3)
\end{align}
on $\M$ and $\subdom$, which is displayed in Figure \ref{fig:spot} (b).  We use samples of $f$ on a  quasi-uniform node set $\Xi\subset\M$ with $N_{\Xi}=32768$ and the restriction of this set $\subXi\subset\subdom$ with $N_{\subXi}=1671$. The evaluation points $X$ and $\subX$ are chosen to densely sample the approximants on the respective domains.

Figure (c) \& (d) show the errors in the MLS approximation using $m=3$ degree polynomials for $\M$ and $\subdom$, respectively.  We see that the approximations in both settings provide good reconstructions of the target. The dimension of the polynomials restricted to the surface were similarly determined by an SVD. However, we found that the point sets were too coarse compared to the resolution of the Spot mesh to reveal the 2-dimensional surface structure and that the dimension of the polynomial space matched the ambient 3-dimensional space in all cases. This suggests an immense flexibility of the proposed MLS method even for non-algebraic surfaces. Instead of breaking down, it essentially works as a standard 3-dimensional MLS approximation based on a point cloud. 

While not presented here, we also observed convergence similar to the previous section for the MLS approximants as the density of the node sets $\Xi$ and $\subXi$ was increased.  While we do not yet have  theoretical guarantees concerning the accuracy and stability of the method in the ambient space, these numerical results are nevertheless promising.  They suggest that the proposed method is robust even in a pre-asymptotic regime and can be employed in a variety of applications. 

A contribution of this paper is to show both theoretically and practically convergence of the MLS in the asymptotic limit of many centers. In this limit, the structure of the manifold will be captured also by the restriction of ambient space polynomials.

\subsection*{Acknowledgements}
The authors wish to thank Daniel Erman and Elizabeth Gross
for helpful comments regarding polynomial ideals, the Hilbert function, and computing  Gr{\"o}bner bases.
They would also like to thank Norm Levenberg for interesting perspectives on the extremal $\mathscr{L}$-function
and feedback on the proof of Lemma \ref{L_Markov}.
Grady Wright was supported by grant DMS-2309712 from the US National Science Foundation.
Thomas Hangelbroek was supported by by grant DMS-2010051 from the US National Science Foundation.

\bibliographystyle{plain}
\bibliography{literature}

\appendix
\section{Higher order Markov inequalities on real algebraic manifolds}\label{S_Markov}
 Any polynomial 
 in $\R[x_1,\dots, x_N]= \mathcal{P}(\R^N)$ has
 a unique analytic extension to the complex ring $\C[z_1,\dots, z_N]=\mathcal{P}(\C^N)$, 
 making the complex variety 
  $$X_{\C}=\{z\in \C^N\mid(\forall j\le N-d)\ \ P_j(z)=0\}$$
  an extension
 of the real
 algebraic variety $X_{\R}$.
 From now on, we denote by $\M_{\R}$ the real algebraic manifold $\M$ (the original object under consideration),
 and by $\M_{\C}$, its complex counterpart.  For  every $x\in \M_{\R}$ 
 there is a neighborhood   in $X_{\C}$ which has no 
 singular points, so $\M_{\R}\subset \M_{\C}$. 

\smallskip
\paragraph{\em Notation} To make clear which distance function is being used, we write $|\cdot|_{\C^n}$ or $|\cdot|_{\R^n}$ for the Euclidean
 norm and $\dist_{\M_{\C}}$ and $\dist_{\M_{\R}}$ for the Riemannian distances on $\M_{\C}$ and $\M_{\R}$, respectively.
 Similarly, to denote balls, we use  $B(x,t;\, \C^n)$, $B(x,t;\, \R^n)$ for Euclidean balls and 
 $B(x,t;\, \M_{\C})$ and $B(x,t;\,\M_{\R})$ for Riemannian balls. We denote  the $\sup$ norm  over a set $A$ by
$\|p\|_{A}:=\sup_{z\in A} |p(z)|$.
 
\smallskip

For this setup,  the maps $\Exp_x: B(0,r_{\Omega})\to B(x,r_{\Omega})$,
with $x\in\Omega$
 are analytic. 
 This follows essentially from \cite[Prop 10.5]{He}
although a direct proof  using analytic dependence on initial conditions is straightforward.
Moreover, each $\Exp_x$  has a complex extension
$ \Exp^{\C}_{x}: B(0,r_x;\, \C^d)\to \M_{\C}\subset \C^N$ with $0<r_x\le r_{\Omega}$.
    This is the exponential map on the complex manifold $\M_{\C}$. 
For  a compact subset $\Omega\subset \M_{\R}$, there is a constant  
 $0<\tilde{r}_{\Omega}\le r_{\Omega}$ so that for any $x\in\Omega$
 and any $r\le \tilde{r}_{\Omega}$,
    $$\Exp^{\C}_{x}: B(0,r;\, \C^d)\to \M_{\C}\subset \C^N$$
    is well-defined, injective, holomorphic, and 
    $\Exp^{\C}_{x}\bigl(B(0,r;\, \C^d)\bigr)\cap \R^N= \Exp^{\C}_{x}( B(0,r;\ \R^d)) $.
%
%

      Because $\Exp^{\C}_{x}$ is a diffeomorphism, we have equivalence between metrics
    on $\C^d$ and $\M_{\C}$; i.e., there are
    constants $0<\widetilde{\Gamma}_1\le \widetilde{\Gamma}_2$ so that 
    $$\widetilde{\Gamma}_1 |\zeta - \xi|_{\C^d} \le \dist_{\M_{\C}}(\Exp^{\C}_x(\zeta),\Exp^{\C}_x(\xi))
    \le \widetilde{\Gamma}_2 |\zeta-\xi|_{\C^d}.$$
    
    \begin{lemma}
\label{L_first_covariant_Markov}
Let $\M_{\R}$ be a an algebraic manifold and
$\Omega\subset \M_{\R}$ a compact subset.
For any $k\in\N$
there is a constant $C_1$, so that for any
$x\in\Omega$, 
any $y\in \Omega\cap B(x,\tilde{r}_{\Omega}/2;\, \R^N)$
with $y = \Exp_{x}(\nu)$ ,
and any polynomial $p\in \mathcal{P}(\M)$,  
$$\|\nabla^k p(y)\|_{T_k^{0}\M_{y}}  
\le
{C_1}{R^{-k}}
\|p\|_{\M_{\C}\cap B(y,R;\, {\C^N})}
$$
holds for all $0<R<\tilde{r}_{\Omega}/2$.
\end{lemma}
We remark that the following proof (and hence the lemma) holds if we replace polynomials  by analytic functions.
Similarly, the algebraic manifold can be replaced by an analytic manifold. This demonstrates that the importance of working on an algebraic manifold is not to obtain this estimate.
\begin{proof} 
Fix $R<\tilde{r}_{\Omega}/2$ 
and without loss, assume $R\le 1$. 
Set  $\tilde{R} = 
{R}/\bigl({\Gamma_2(1+\sqrt{d})}\bigr)$,
and
write the polynomial $p$ as 
$p({\Exp_x}(z)) = P(z)$, where $P\in C^{\infty}(B(0,\tilde{R};\, \R^d))$
is analytic. 
Then (\ref{eq:expmap}) ensures that, for $y\in \Omega\cap B(x,\tilde{R};\, \M)$,
  $$
 \|\nabla^k p(y)\|_{T_k^{0}\M_{y}} 
 \le 
 \|p\|_{C^k(\Exp_x(B(\nu,\tilde{R}))} 
 \le
 C_2\|P\|_{C^k(B(\nu,\tilde{R};\, \R^d))}.
 $$
The complex extension 
$\tilde{P} =p\circ \Exp^{\C}_{x}:B(0,R;\, \C^d)\to \C$ 
is also analytic,
so
we have
$$
 \|\nabla^k p(y)\|_{T_k^{0}\M_{y}} 
 \le C_2 \max_{\alpha\le k}
 \| D^{\alpha}P \|_{B(\nu,\tilde{R},\, \R^d)}
= C_2 \max_{\alpha\le k}
\max_{\zeta\in  B(\nu,\tilde{R},\, \R^d) }
 | D^{\alpha}\tilde{P}(\zeta)|.
 $$
 Fix $\zeta\in B(\nu,\tilde{R},\, \R^d) =\R^d\cap B(\nu,\tilde{R};\, \C^d)$. 
 By the Cauchy integral formula, we can write
 $$
 D^{\alpha}\tilde{P}(\zeta)
 = \frac{\alpha!}{(2\pi i)^d} \int_{|\xi_1-z_1|=\tilde{R}}\dots \int_{|\xi_d-z_d|=\tilde{R}} 
 \frac{\tilde{P}(\xi_1,\dots\xi_d)}{(\xi_1-\zeta_1)^{\alpha_1+1}\dots (\xi_d-\zeta_d)^{\alpha_d+1}}\diff \xi_1\dots \diff \xi_d .
 $$
 This gives Cauchy's inequality  $|D^{\alpha} \tilde{P} (z) |\le \tilde{R}^{-|\alpha|}\max_{\xi\in \mathbb{D}(z,\tilde{R})}|P(\xi)|$,
 where we have employed the polydisk
 $\mathbb{D}(z,\tilde{R}):=\{z\in\C^d\mid (\forall j\le d)\, |z_j|\le \tilde{R}\}$.
Since the polydisk is contained in a ball, 
$\mathbb{D}(\zeta,\tilde{R})\subset B(\zeta,\tilde{R}\sqrt{d};\, \C^d)\subset B(\nu,\tilde{R}(1+\sqrt{d});\, \C^d)$,
we  have
 $$|D^{\alpha} \tilde{P} (\zeta) |
\le 
\alpha!
{\tilde{R}^{-|\alpha|}}
\max_{\xi\in B(\nu,\tilde{R}(1+\sqrt{d}) ;\, \C^d)} |\tilde{P}(\xi)|.
$$
The result follows
by 
replacing
 $\tilde{R}^{-|\alpha|}$ by ${R}^{-|\alpha|}({\Gamma_2(1+\sqrt{d})})^{|\alpha|}$ and subsequently 
enlarging the constant $C = C_2(\Gamma_2(1+\sqrt{d}))^{|\alpha|} $
to obtain
$$\|\nabla^k p(y)\|_{{T_k^{0}\M_{y}}} 
\le
{C}{R^{-|\alpha|}}
\max_{\xi \in \Exp_x(B(\nu, \tilde{R}(1+\sqrt{d});\, \C^d))} |p(\xi)|.$$
By assumptions on $\tilde{R}$ and $y$, we have
 $$\Exp^{\C}_x(B(\nu,\tilde{R}(1+\sqrt{d});\, \C^d))\subset 
 B(y, \Gamma_2 \tilde{R}(1+\sqrt{d});\,\M_{\C})
 \subset B(x, \tilde{r}_{\Omega};\, \M_{\C})$$
so $\|\nabla^k p(y)\|_{T_k^{0}\M_{y}}  
\le
{C}{R^{-k}} \max_{z \in B(y, R;\, \M_{\C})} |p(z)|$ holds. 
Finally,
by the metric comparison
$|v-y|\le \dist_{\M_{\C}} (y,v)$
we have $ B(y, R;\, M_{\C})\subset \M_{\C}\cap B(y,R;\, \C^N) $ and the lemma follows.
\end{proof}
 At this point, the issue of obtaining
the higher order Markov inequality is precisely the same as first order version given in
handled in  \cite[Main Theorem (1) $\to$ (2)]{bos:etal:1995},
which we follow.
 We include this proof
for the sake of completeness.

We wish to refine the sup norm  from Lemma \ref{L_first_covariant_Markov}, so that  it is taken
 over a real neighborhood  $K\subset \M_{\R}$ of $y$, rather
than the neighborhood  $\M_{\C}\cap B(y,r;\, \C^N) $.
For this, we use a Bernstein-Walsh inequality
\cite[(B-W)]{bos:etal:1995} (see also \cite[Ch. 5]{K})
of the form
$$|p(z)| \le  \|p\|_K \exp^{\deg(p) u_K(z)},$$
where $u_K$ is the $\mathscr{L}$-extremal function, which
for compact $K$ is equal to 
\begin{equation}
\label{extremal}
u_K(z) =\max\left(0, \sup\left\{\frac{1}{\deg{p}} \log\frac{|p(z)|}{\|p\|_K} \mid p\in\mathcal{P}(\C^N),\  \deg(p)>0\right\}\right).
\end{equation}
The extremal $\mathscr{L}$-function has a more complicated definition when $K$ is not compact,
although for any two sets $K\subset K'$ we have $u_{K'} \le u_K$, and for any set $K$, $u_K = u_{\hat{K}}$,
where $\hat{K}= \{z\in \C^N\mid \forall p\in \mathcal{P}(\C^N)\, |p(z)|\le \|p\|_K\}$. In particular, $u_{K'} = u_K$ if $K=\mathrm{cl}(K')$.

 Although definition (\ref{extremal})  makes the Bernstein-Walsh inequality a tautology, 
 finding bounds for $u_K$, or even showing that
$u_K$ is locally bounded  is a challenge. 
A major technical accomplishment of \cite{bos:etal:1995}
is that under certain circumstances, specifically, for certain compact sets $K\subset \M_{\R}$ and  neighborhoods
$U\subset \M_{\C}$, 
the quantities $\|u_K\|_{U}$
 are  bounded.

The main tool to bound the $\mathscr{L}$-function is \cite[Corollary 1.9]{bos:etal:1995}, 
summarized as follows. 
\begin{lemma}[Bos-Levenberg-Milman-Taylor]
For an
analytic map $$\sigma:B(0,1;\, \C^d)\to \M_{\C}$$ having full rank (i.e., an immersion), 
with $\sigma(B(0,1;\, \R^d))\subset \R^N$,
there are $\sigma$-dependent constants $\tilde{C}$ and $\epsilon_0<1/8$ so that 
 any point $x^0\in \sigma(B(0,1/4;\, \R^d))\cap \R^N$ and 
 for any $\epsilon<\epsilon_0$, the real neighborhood
$$K_{\epsilon}(x_0):= [\sigma (B(0,1/4;\, \C^d)\cap\R^N]\cap B(x_0,\epsilon;\, \R^N)$$
has $\mathscr{L}$-function which obeys
\begin{equation}
u_{K_{\epsilon}(x_0)} (z)\le
 \tilde{C} \frac{ 
  \bigl\| u_{\sigma(B(0,1/3;\, \C^d))}\bigr\|_{\sigma(B(0,1;\, \C^d) }
  }{\epsilon}
\dist_{\C^N}(z, K_{\epsilon}(x_0)).
\end{equation}
for $z\in  \sigma (B(0,1;\, \C^d))\cap B(x_0,\epsilon/2;\, \C^N)$.
\end{lemma}

\begin{note}
Note that $ \bigl\| u_{\sigma(B(0,1/4;\, \C^d))}\bigr\|_{\sigma(B(0,1;\, \C^d) )}$ is also a $\sigma$-dependent
constant. 
The fact that it is finite follows from the earlier work of Sadullaev \cite{Sadullaev}, which demonstrates a bound for 
the $\mathscr{L}$-extremal functions for complex neighborhoods. Specifically, \cite[Theorem 2.2]{Sadullaev}
shows that for some compact set $K\subset \M_{\C}$, $u_K$ is locally bounded on $\M_{\C}$, 
while \cite[Proposition 2.1]{Sadullaev}
shows that if it $u_K$ is locally bounded for one compact set 
$K\subset \M_{\C}$, then $u_{\tilde{K}}$ is locally bounded
for any compact set $\tilde{K}$ which is sufficiently large; 
namely, for any  compact set which is {\em non-pluripolar} in $A$.
The finiteness of $\bigl\| u_{\sigma(B(0,1/3;\, \C^d))}\bigr\|_{\sigma(B(0,1;\, \C^d) )}$  
follows from the fact that 
$\sigma(B(0,1/3;\, \C^d))$ contains a compact, non-pluripolar subset, as consequence, e.g.,
of \cite[Lemma 1.7]{bos:etal:1995}.
\end{note}
\begin{proof}[Proof of Lemma \ref{L_Markov}]
For every  $x\in\Omega$, rescale  $\Exp_x^{\C}$ 
as  $\sigma_x:B(0,1;\, \C^d)\to \M_{\C}$ by 
$\sigma_x(\zeta) =\Exp^{\C}_x(\tilde{r}_{\Omega}\zeta)$.
Then  \cite[Corollary 1.9]{bos:etal:1995}
guarantees
existence of constants $0<\epsilon_{x}<1/8$ and $C_x$
so that for each $\epsilon<\epsilon_{x}$ and  
each $y\in \sigma_x(B(0,1/4);\, \C^d)\cap\R^N$
$$u_{K_{\epsilon} (y)}(z)\le  \frac{C_x}{\epsilon} \dist(z,K_{\epsilon}(y)).$$
whenever  $z\in \sigma_x(B(0,1))\cap B(y,\epsilon/2;\, \C^N)$.

Now cover $\Omega$ by a finitely many neighborhoods 
$\sigma_{x_i}(B(0,1/4;\, \C^d))= B(x_i, \tilde{r}_{\Omega}/4;\ \M_{\C})$ using $x_i\in \Omega$.
 Let $\epsilon_0= \min \epsilon_{x_i}$ 
 and let $C_2 = \fix{\max\; } C_{x_i}$.\fixcom{Changed $\min$ to $\max$}
\fix{It follows that  for  any $y\in \Omega$, there is $x_i\in \Omega$
so that $y\in \sigma_{x_i}(B(0,1/4);\, \C^d)\cap\R^N$.
Furthermore for any  $r<\epsilon_0$,
 $K_{r}(y)=
  [\sigma_{x_i} (B(0,1/4;\, \C^d)\cap\R^N]\cap B(y,r;\, \R^N)
  = \M_{\R}\cap B(y,r;\, \R^N)$
 has the property that
 $u_{K_{r}(y)}\le  \frac{C_2}{r} \dist(z,\M_{\R}\cap B(y,r;\, \R^N)).$
 Applying the Bernstein-Walsh inequality gives
 $$
 |p(z)| \le \|p\|_{\M_{\R}\cap B(y,r;\, \R^N) }  \ e^{ \deg(p)\frac{C_2}{r} \dist(z,\M_{\R}\cap B(y,r;\, \R^N))}
 $$
 for any $z\in B(y,r/2;\C^N)$.}
 
  Set $\epsilon_* = \min(\epsilon_0, \tilde{r}_{\Omega}/2)$ and pick $r<\epsilon_*$.
 Note that if $p$ is constant, the inequality holds automatically, so we assume $\deg(p)\ge 1$.
Set $R:= \fix{r}/\max(2,\deg(p))$\fixcom{changed $\epsilon$ to $r$} and note that the conditions for Lemma \ref{L_first_covariant_Markov} are satisfied.
Thus for $y\in \Omega$ 
\begin{eqnarray*}
\|\nabla^k p(y)\|_{T_k^{0}\M_{y}}  
&\le&
{C_1}{R^{-k}} \|p\|_{\M_{\C}\cap B(y,R;\, {\C^N})}\\
&\le& {C_1}{R^{-k}} \|p\|_{\M_{\R}\cap B(y,r;\, \R^N) }  e^{ \deg(p)\frac{C_2}{r} R}\\
&\le& C_3{R^{-k}} \|p\|_{\M_{\R}\cap B(y,r;\, \R^N) }\ .
\end{eqnarray*}
In the second to last line we use the fact that $ \dist(z,\M_{\R}\cap B(y,r;\, \R^N)\le R$,
when $z\in B(y,R;\C^N)$. 
In the last line, we
set $C_3:= C_1 e^{C_2}$.
The assumption on $R$ guarantees that
$
\|\nabla^k p(y)\|_{T_k^{0}\M_{y}}   \le 2^k C_3 (\deg(p))^k{r^{-k}} \|p\|_{\M_{\R}\cap B(y,r;\, \R^N) }
$. By increasing the constant, we have
$
\|\nabla^k p(y)\|_{T_k^{0}\M_{y}}   \le C_4r^{-k} \|p\|_{ B(y,r;\M_{\R}) }
$.

Finally, by taking the supremum over $  B(y,r;\M_{\R}) $,
we have 
$$
\|p\|_{C^k(B(y,r;\M_{\R}))}   \le C_4  r^{-k}\|p\|_{ B(y,2r;\M_{\R})  }
\rightarrow 
\|p\|_{C^k(B(y,r;\M_{\R}))}   \le C_4 e^{C_* m} r^{-k} \|p\|_{ B(y,r;\M_{\R})  }
$$
by applying the doubling estimate (\ref{Fef_Doubling}) of Fefferman and Narasimhan.
 \end{proof}

\section{The dimension of \texorpdfstring{$\mathcal{P}_m(\M)$}{polynomials restricted to affine varieties}}\label{S_ideal}
The algebra of polynomials on $\M\subset\R^N$ can be expressed as a quotient 
$\mathcal{P}(\M) = \mathcal{P}(\R^N)/I(\M)$,
 with the help of the 
ideal 
$$I(\M):= \{p\in \mathcal{P}(\R^N)\mid p|_{\M} = 0\}.$$
The map $m\mapsto \dim\bigl( \mathcal{P}_m(\M)\bigr)$ is called the {\em Hilbert function} of $I(\M)$.
 In this case,
$\mathcal{P}_m(\M) =  \mathcal{P}_m(\R^N)/I_m(\M)$, where 
$I_m(\M): = I(\M)\cap \mathcal{P}_m(\R^N)$.
Thus 
$$\dim(\mathcal{P}_m(\M)) = \dim\bigl( \mathcal{P}_m(\R^N)/I_m(\M)\bigr)= \dim( \mathcal{P}_m(\R^N))- \dim (I_m(\M)).$$

\subsubsection*{Algebraic varieties of codimension 1}
This observation is enough to cleanly handle the case of algebraic  hypersurfaces
 $\M = \{x\in \R^N\mid P(x)=0\}$ 
  in $\R^N$
 \revision{ which satisfy the extra condition 
$$I(\M) = \{ q P\mid q\in \mathcal{P}(\R^N)\}.$$ 
In other words,
$I(\M)$ can be expressed as the ideal generated by $P$.}
\revisioncom{We have streamlined the condition on the variety.}
 In other words,
$I(\M)$ can be expressed as the ideal generated by $P$.
Consequently, $I_m(\M) = \{qP\mid \deg(qP) \le m\}$. 
Writing $\ell:=\deg(P)$, we then have, for $m\ge k$, 
that $\dim(I_m(\M)) = \dim(\mathcal{P}_{m-\ell}) (\R^N)$,
while for $m<k$,  $\dim(I_m(\M)) = 0$. Thus,
\begin{equation}
\label{Hypersurface_dimension}
\dim(\mathcal{P}_m(\M)) =
\begin{pmatrix}  N+m\\ N\end{pmatrix} - \begin{pmatrix} N+m-\deg(P)\\ N\end{pmatrix}
\end{equation}
\revision{with the convention that the binomial coefficient $\begin{pmatrix}a\\b\end{pmatrix}$ vanishes if $b>a$.}

\subsubsection*{Higher codimension} The situation for $k=N-d\ge 2$ (varieties of  codimension greater than one)
is a bit more complicated.
We give a roadmap which treats the scenario where $\M$
is the nullset of $k$ polynomials: $\M = \{x\in \R^N\mid  (\forall j\le k)\,P_j(x) =0\}$. The interested reader will find  precise details and proofs in \cite{Cox:etal:2015}.

We assume that the  ideal  $I(\M)=\{p\in \mathcal{P}(\R^N)\mid p|_{\M}=0\}$ is equal to
the ideal generated by the polynomials which have defined it:
\begin{equation}
\label{real_radical}
I(\M) =\langle P_1,\dots,P_k\rangle:= \Bigl\{\sum_{j=1}^k q_j P_j\mid q_j\in \mathcal{P}(\R^N)\Bigr\}.
\revisioncom{We point out that for a complex variety $\{x\in\C^N\mid P_j(x)=0\}$,
(\ref{real_radical}) follows when $\langle P_1,\dots,P_k\rangle$ is radical by Hilbert's
{\em Nullstellensatz}. 
 The situation is more complex for real algebraic varieties, although
there does exist a {\em real Nullstellensatz}
following work of Krivine in the 1960s. 
Roughly, 
condition (\ref{real_radical} holds if and only if $I(\M)$ is the {\em real radical} $\sqrt[\R]{I}$ of $I=\langle P_1,\dots,P_k\rangle$.
The real radical has a rather
technical definition -- we have elected not to include 
this digression
in the manuscript, but include it here for the reviewers.
}
\end{equation}

%

In this case,  $\dim(\mathcal{P}_m(\M)) =\dim(\mathcal{P}_m(\R^N))-
 \dim I_m(\M)$
  is obtained as before from 
  the Hilbert function of $I(\M)$.
  To facilitate this calculation,  a surrogate ideal whose Hilbert function is more easily computed is used.

\begin{enumerate}
\item
The {\em initial} ideal $\tilde{I}$ 
of $I(\M)$ is the ideal generated by the lead terms of $I(\M)$
under an admissible ordering on $\Z_+^N$. For this, the
lexicographic ordering suffices, although many other orderings
will do,  each may produce different lead terms and a different initial ideal -- see \cite[Chapter 2.2]{Cox:etal:2015}.
 However, regardless of this ordering,
 \cite[Ch.\ 9.3,\, Proposition 4]{Cox:etal:2015} guarantees that  the Hilbert function of $\tilde{I}$ is the same that of $I(\M)$.
 \item
The initial ideal
  is generated by finite sequence of monomials,  $\tilde{I} = \langle x^{\alpha_1},\dots ,x^{\alpha_M}\rangle$,
  which occur as the lead terms of a Gr{\"o}bner basis of $I(\M)$.
 There are  algorithms which produce a Gr{\"o}bner basis from an initial list of generators of an ideal;
 many have been implemented symbolically (see below for an example).
 \item
 The quantity
$\dim\tilde{I}_m$ can be calculated easily because $\tilde{I}$ is monomial ideal.
Indeed,  $\tilde{I} =\mathrm{span}\{x^{\alpha}\mid \alpha\in \mathcal{I}:= \bigcup_{j=1}^n( \alpha_j +\Z_+^N)\}$,
so the Hilbert function is
$$\dim I_m(\M)=\dim\tilde{I}_m=\revision{\mathrm{card}}\Bigl\{\alpha\in\bigcup_{j=1}^n( \alpha_j +\Z_+^N)\mid |\alpha|\le m\Bigr\}.$$
\end{enumerate}
\subsubsection*{A basis for $\mathcal{P}_m(\M)$}
An interesting fact is that the initial ideal yields a basis for the space $\mathcal{P}_m(\M)$. 
Specifically,  the result
\cite[Theorem 2.6]{Gr}, 
shows that the {\em standard basis} for $\mathcal{P}(\M)$ consists of 
all monomials $x^{\alpha}\notin \tilde{I}$, and so 
$$\{x^{\alpha}\notin \tilde{I}\mid |\alpha|\le m\}$$
is a basis for $\mathcal{P}_m(\M)$. Another formulation of this result is
\cite[Theorem 1.6.12]{Hibi2013}.
\paragraph{\em Example} 
Consider the 1-dimensional manifold $SO(2)\subset \R^4$, where 
$$\begin{pmatrix} a&b\\c& d\end{pmatrix}\in SO(2)$$
The ideal associated with this variety is 
$$I=\langle a^2 +c^2-1,b^2+d^2-1, ab +cd ,ad-bc-1\rangle.$$
A Gr{\"o}bner basis for this ideal is $\{b+c, a-d, c^2+d^2-1\}$.
The initial ideal is generated by the lead terms: $a,b$ and $c^2$, which corresponds to 
multi-indices 
$$\alpha_1 =(1,0,0,0), \quad \alpha_2=(0,1,0,0)\quad \text{and} \quad \alpha_3 =(0,0,2,0).$$
One may count, by hand that $\dim(I_0)=0$, $\dim(I_1) =\mathrm{card}\{\alpha_1,\alpha_2\}=2$, and even
 $\dim(I_2) =  \mathrm{card}\bigl(\{ \alpha_3\} \cup \{\alpha_1+\gamma\mid  |\gamma|\le 1\}\cup\{\alpha_2+ \gamma\mid |\gamma|\le 1 \}\bigr) =10$.
In general, the fact that  $\dim(I_m) =2 \begin{pmatrix}m+3\\4\end{pmatrix} 
 - 2\begin{pmatrix}m+1\\4\end{pmatrix}
 +\begin{pmatrix}m\\4\end{pmatrix}  $ follows  by inclusion/exclusion. Since $\mathcal{P}_m(\R^4) = \begin{pmatrix}m+4\\4\end{pmatrix}$,
 we have
 $$\mathcal{P}_m(SO(2))=  \begin{pmatrix}m+4\\4\end{pmatrix}-2\begin{pmatrix}m+3\\4\end{pmatrix} 
 +2\begin{pmatrix}m+1\\4\end{pmatrix}
 -\begin{pmatrix}m\\4\end{pmatrix} =2m+1.$$
 Furthermore, we see that  $x^\alpha$ with $\alpha_1=0$, $\alpha_2=0$, $\alpha_3<2$  are
 the monomials not in $\tilde{I}$. Thus
 $\{c d^j:j=0,1,\dots\}\cup \{d^j\mid j=0,1,\dots\}$ is 
 a basis for $\mathcal{P}(\M)$; using the standard parametrization $a,d= \cos\theta$, $c=-b = \sin\theta$ for $SO(2)$, this
 corresponds to the trigonometric polynomial basis 
  $\{\sin (\cdot) \cos^j (\cdot)\mid j\ge 0\}\cup\{\cos^j (\cdot)\mid j\ge 0\}$.
\subsubsection*{Symbolic computing} There are a variety of computer algebra systems that
perform symbolic computing with polynomial ideals, like
computing Gr{\"o}bner bases.
An introduction to several of these platforms is presented in \cite[Appendix C]{Cox:etal:2015}. 
For example,
we can use {\tt Macaulay2} to analyze $SO(2)$:
\begin{itemize}
\item {\tt R = QQ[a,b,c,d]} generates the polynomial ring (along with the monomial order),
\item {\tt I = ideal$(a^2 +c^2-1,b^2+d^2-1, a*b +c*d ,a*d-b*c-1)$},  generates the ideal,
\item {\tt gens gb I} generates a Gr{\"o}bner basis,
\item {\tt leadTerm(I)} 
gives the monomial generators for $\tilde{I}$.
\end{itemize}
An issue which may arise when treating the kinds of algebraic varieties considered in this article is that
some of  these platforms  work with polynomial rings over discrete fields, like the rationals. Thus some
caution may be needed, even when working with varieties defined by polynomials with integer coefficients.
\end{document}